\documentclass[11pt]{article}

\usepackage{amsmath,amsthm,amssymb,enumerate,dsfont}
\usepackage{algorithm}
\usepackage{algorithmic}[1]
\usepackage{float}
\usepackage{subfig}
\usepackage{titlesec}
\usepackage{scalerel}
\usepackage{epstopdf}
\usepackage{graphicx}
\usepackage{xspace}
\usepackage{color}
\usepackage{hyperref}
\usepackage[title]{appendix}
\usepackage[margin=1in]{geometry}

\usepackage[colorinlistoftodos,prependcaption,textsize=tiny]{todonotes}

\newcommand{\RR}{\mathbb{R}}

\newcommand{\ra}{\rightarrow}

\newcommand{\ie}{{\it i.e.}}

\newtheorem{lem}{Lemma}
\newtheorem{thm}{Theorem}
\newtheorem{coro}{Corollary}
\newtheorem{defi}{Definition}
\newtheorem{prop}{Proposition}

\newtheorem{rmk}{Remark}

\begin{document}
\title{\textbf{Universal Barrier is $n$-Self-Concordant}}
\author{Yin Tat Lee\thanks{University of Washington.  E--mail: {\tt yintat@uw.edu}} \and Man--Chung Yue\thanks{The Hong Kong Polytechnic University.  E--mail: {\tt manchung.yue@polyu.edu.hk}}}
\date{ }
\maketitle
\begin{abstract}
This paper shows that the self-concordance parameter of the universal barrier on any $n$-dimensional proper convex domain is upper bounded by $n$. This bound is tight and improves the previous $O(n)$ bound by Nesterov and Nemirovski. The key to our main result is a pair of new, sharp moment inequalities for $s$-concave distributions, which could be of independent interest.

\end{abstract}

\bigskip


\bigskip

\section{Introduction}

In a seminal work~\cite{nesterov1994interior}, Nesterov and Nemirovski developed a theory of interior point methods for solving general nonlinear convex constrained optimization problems. A central object of their theory is the \emph{self-concordant barrier} for the feasible region. Roughly speaking, a self-concordant barrier on a proper convex domain\footnote{A \emph{convex domain} is a convex set with non-empty interior. A convex set is said to be \emph{proper} if it does not contain any 1-dimensional affine subspace.} $K$ is a convex function that satisfies certain differential inequalities and blows up at the boundary $\partial K$ (see Section~\ref{sec:prelim} for the precise definition). Associated with any self-concordant barrier is the \emph{self-concordance parameter} $\nu \ge 0$. The importance of self-concordant barriers lies in the fact that the path-following interior point method developed in~\cite{nesterov1994interior} approximately solves a convex constrained optimization problem in $O(\sqrt{\nu} \log{(1/\epsilon)})$ iterations if the feasible region has a $\nu$-self-concordant barrier.

It is then natural to ask whether one can construct a self-concordant barrier for arbitrary proper convex domain and, if yes, what the self-concordance parameter $\nu$ is. The first result along this direction was given by Nesterov and Nemirovski~\cite{nesterov1994interior}: they constructed a self-concordant barrier for general proper convex domain $K\subseteq \RR^n$, the so-called \emph{universal barrier}, and proved that it is $O(n)$-self-concordant. They also showed that any self-concordant barrier of $n$-dimensional simplex or hypercube must have self-concordance parameter at least $n$, see~\cite[Proposition 2.3.6]{nesterov1994interior}. Hence, their self-concordance bound is \emph{order-optimal}.

Another self-concordant barrier, the \emph{entropic barrier}
, was recently studied by Bubeck and Eldan~\cite{bubeck2015entropic}. Exploiting the geometry of log-concave distributions and duality of exponential families, Bubeck and Eldan~\cite{bubeck2015entropic} proved that the entropic barrier satisfies the self-concordance parameter guarantee $\nu \le n + O(\sqrt{n \log{n}})$ for $n\ge 80$, thus improving the result of Nesterov and Nemirovski~\cite{nesterov1994interior}.

When the proper convex domain $K$ is a cone, the situation is clearer. Indeed, the \emph{canonical barrier}, introduced by Hildebrand~\cite{hildebrand2014canonical} and independently by Fox~\cite{fox2015schwarz}, is an $n$-self-concordant barrier of proper convex cones with non-empty interior. Furthermore, using a result of G{\"u}ler~\cite{guler1996barrier}, Bubeck and Eldan~\cite{bubeck2015entropic} showed that both the universal barrier and the entropic barrier are also $n$-self-concordant on proper convex cones. These results confirmed a conjecture\footnote{See the discussions in~\cite{bubeck2015entropic} and~\cite{fox2015schwarz}.} made by G{\"u}ler which asserted that, for any proper convex cone in $\RR^n$, there always exists a self-concordant barrier whose self-concordant parameter is at most $n$.

This paper completes the picture by settling the same question in the more general case of proper convex domains. We show that the universal barrier is $n$-self-concordant on any proper convex domain $K\subseteq \RR^n$ for $n\ge 1$. This does not only improve the results of~\cite{nesterov1994interior} and~\cite{bubeck2015entropic} but is also tight in view of the above-mentioned lower bound on the self-concordance parameter. The key to this result is a pair of new, sharp moment inequalities for $s$-concave distributions (see Section~\ref{sec:prelim} for the definition of $s$-concavity), which could be of independent interest. One of these inequalities is a generalization of~\cite[Lemma 2]{bubeck2015entropic}.

We should emphasize that all these bounds on the self-concordant parameters of different barriers do not immediately yield polynomial-time complexity result for convex programming problems. The iteration complexity $O(\sqrt{\nu} \log{(1/\epsilon)})$ counts only the number of iterations of the path-following algorithm, whereas the overall complexity depends also on the costs of computing the gradient and the Hessian of the barrier for the feasible region. The problem of constructing self-concordant barriers with (nearly) optimal self-concordance parameter and efficiently computable gradient and Hessian remains largely open. A recent breakthrough was obtained in the context of polytopes by Lee and Sidford~\cite{lee2014path}. However, our result does find applications in some online learning problems where the quality of solutions produced by certain algorithms depend on the self-concordance parameter~\cite{abernethy2009beating, narayanan2010random}.

The rest of the paper is organized as follows. Section~\ref{sec:prelim} collects some necessary background and preparatory results. The optimal self-concordance bound of the universal barrier, which is the main result of this paper, will be proved in Section~\ref{sec:uni-bar}. Section~\ref{sec:prop} provides the proofs of the pair of moment inequalities used for proving the main result. 

\subsection{Notations}
We adopt the following notations throughout the paper. Given a set~$S$, we denote by $\mathrm{cl}(S)$, $\mathrm{int}(S)$ and $\partial S = \mathrm{cl}(S)\! \setminus\! \mathrm{int}(S)$ the closure, interior and boundary of $S$, respectively. The indicator function of $S$ is denoted by $\mathds{1}_S$, \ie, $\mathds{1}_S(t) = 1$ if $t\in S$ and $\mathds{1}_S(t) = 0$ otherwise. We denote by $\mathrm{Vol}_k (S)$ the $k$-dimensional Lebesgue measure of $S$. For any function $\psi$, the $i$-th directional derivative of $\psi$ at $x$ along the direction $h$ will be denoted by $D^i \psi(x)[h,\dots, h]$. For any distribution on $\RR$ with density $p$, we denote by $\mathrm{Supp}(p)$ the support of the distribution, \ie, $\mathrm{Supp}(p) = \mathrm{cl} \left(\{ t\in \RR: p(t) > 0 \}\right)$. The Dirac delta distribution at $t$ will be denoted by $\delta_t$.

\section{Preliminaries}\label{sec:prelim}

\subsection{The Universal Barrier}
A \emph{convex domain} is a convex set with non-empty interior. A convex set is said to be \emph{proper} if it does not contain any 1-dimensional affine subspace. Throughout the paper, if not specified, $K $ will always denote a proper convex domain in $\RR^n$. As usual, a \emph{convex body} refers to a compact convex domain. The following definitions are standard~\cite{nesterov1994interior}.
\begin{defi}
A function $\phi :\mathrm{int}(K)\ra \RR$ is said to be a \emph{barrier} on $K$ if 
\begin{equation*}
\phi(x) \ra +\infty \quad\text{as}\quad x\ra \partial K.
\end{equation*}
\end{defi}
\begin{defi}
A three times continuously differentiable convex function $\phi$ is said to be \emph{self-concordant} on $K$ if for any $x\in \mathrm{int}(K)$ and $h\in \RR^n$,
\begin{equation}\label{ineq:sc}
\left| D^3 \phi(x) [h,h,h] \right| \le 2 \left( D^2 \phi(x)[h,h] \right)^{\frac{3}{2}}.
\end{equation}
If, in addition to~\eqref{ineq:sc}, $\phi$ satisfies that for any $x\in \mathrm{int}(K)$ and $h\in \RR^n$,
\begin{equation}\label{ineq:nu_sc}
\left| D \phi(x) [h] \right| \le  \left(\nu\cdot D^2 \phi(x)[h,h] \right)^{\frac{1}{2}},
\end{equation}
then $\phi$ is said to be \emph{$\nu$-self-concordant}.
\end{defi}
The main contribution in this paper concerns the so-called \emph{universal barrier} introduced by Nesterov and Nemirovski~\cite{nesterov1994interior}.
\begin{defi}
The \emph{universal barrier} of $K$ is defined as the function $\phi:\mathrm{int}(K) \ra \RR$ given by 
\begin{equation*}
\phi(x) = \log \mathrm{Vol}_n\left( K^\circ (x) \right),
\end{equation*}
where $K^\circ (x) = \left\{ y \in \RR^n : y^T(z-x) \le 1, \forall z\in K \right\}$ is the polar set of $K$ with respect to $x$.
\end{defi}
\noindent It is well-known that the universal barrier is $O(n)$-self-concordant~\cite[Theorem 2.5.1]{nesterov1994interior}. As we will see in the Section~\ref{sec:uni-bar}, the bound $O(n)$ can be improved to exactly $n$.

\subsection{Probabilistic Tools}
Since all distributions considered in this paper are absolutely continuous with respect to the Lebesgue measure, we identify a distribution with its density. For any distribution~$p$ on $\RR$, we denote its mean by $\mu_1(p)$ and the second and third moments about the mean by $\mu_2^2(p)$ and $\mu_3^3(p)$, respectively, \ie,
\begin{align*}
\mu_1(p) &= \int_{-\infty}^\infty t p(t) dt,\\
\mu_2^2(p) &= \int_{-\infty}^\infty (t - \mu_1(p) )^2 p(t) dt \quad\text{and} \\
\mu_3^3(p) &= \int_{-\infty}^\infty (t - \mu_1(p) )^3 p(t) dt.
\end{align*}

The following type of distributions on $\RR$ is particularly important in this paper.
\begin{defi}
Let $L\subseteq \RR^n$ be any convex body and $h\in \RR^n$. The \emph{marginal distribution of the convex body $L$ along the direction $h$}, denoted by $p\left(L,h;\cdot\right)$, is the distribution on $\RR$ given by, for any $t\in \RR$,
\begin{equation*}
p\left(L,h;t\right) = \frac{\mathrm{Vol}_{n-1} \left(\left\{ y\in L: y^Th = t \right\}\right) }{\mathrm{Vol}_n (L)}.
\end{equation*}
\end{defi}
\noindent Note that the polar set $K^\circ(x)$ with respect to any $x\in \mathrm{int}(K)$ is a convex body. Therefore, we can talk about its marginal distributions. Interestingly, the directional derivatives of the universal barrier on $K$ at $x$ can be expressed in terms of moments of the marginal distribution of the polar set $K^\circ (x)$. The following formulas\footnote{The formulas in \cite[p.~52]{nesterov1994interior} contain some minor sign errors. Here we present the corrected ones.} can be found in \cite[p.~52]{nesterov1994interior}.
\begin{lem}\label{lem:direc_deri}
Let $x\in \mathrm{int}(K)$ and $h\in\RR^n$ be given. Let $p = p\left(K^\circ (x),h;\cdot\right)$. Then we have that
\begin{align*}
D \phi(x)[h] & =  (n+1) \mu_1(p), \\
D^2 \phi(x)[h,h] & = (n+1)(n+2)\mu_2^2(p) + (n+1) \mu_1^2(p)\text{ and }\\
D^3 \phi(x)[h,h,h] & =  (n+1)(n+2)(n+3) \mu_3^3(p) + 6(n+1)(n+2)\mu_2^2 (p)\mu_1(p) + 2(n+1) \mu_1^3(p).
\end{align*}
\end{lem}

Next, we recall the definition of $s$-concave distributions~\cite{brazitikos2014geometry}.
\begin{defi}\label{defi:s-concave}
A distribution $p$ on $\RR$ is said to be \emph{$s$-concave} if for any $\lambda \in [0,1]$ and $t_1,t_2\in \RR$, it holds that
\begin{equation}\label{ineq:s-concave}
p( \lambda t_1 + (1-\lambda)t_2 ) \ge \left( \lambda \left(p(t_1)\right)^s + (1-\lambda) \left(p(t_2)\right)^s \right)^{\frac{1}{s}}.
\end{equation}
For the case $s = 0$, $s= -\infty$ and $s=+\infty$, the right-hand side of \eqref{ineq:s-concave} becomes $\left(p(t_1) \right)^\lambda \left(p(t_2) \right)^{1-\lambda}$, $\min\{p(t_1), p(t_2)\}$ and $\max\{p(t_1), p(t_2)\}$.
\end{defi}
\noindent Note that 0-concave distributions are nothing but log-concave distributions.

We pause to provide some intuitions for the $O(n)$ bound on the self-concordance parameter of Nesterov and Nemirovski~\cite{nesterov1994interior} and explain why improvement is possible. First, it is a fact in convex geometry that the width of a convex body $L\subseteq \RR^n$ along any direction $h$ is of the order $O\left(n\cdot \mu_2\left( p\left(L,h; \cdot \right)\right)\right)$. Lemma~\ref{lem:direc_deri} then implies that $\phi$ satisfies inequality~\eqref{ineq:nu_sc} with $\nu = O(n)$. Second, the Pr{\'e}kopa–-Leindler inequality implies that $p(L,h; \cdot )$ is a log-concave distribution. Combining this with another convex-geometric fact that the third moment of any log-concave distribution is bounded by its second moment, we can deduce inequality~\eqref{ineq:sc} from Lemma~\ref{lem:direc_deri}. Our improvement is made possible by the observation that any marginal distribution $p(L,h; \cdot )$ is actually $\tfrac{1}{n-1}$-concave, a stronger property than the log-concavity. This observation follows immediately from the \emph{Brunn's concavity principle}.

%

\begin{thm}[Brunn's Concavity Principle, {\cite[Theorem 1.2.2]{brazitikos2014geometry}}]\label{thm:brunn}
Let $L$ be a convex body in $\mathbb{R}^{n}$ and $F$ be a $k$-dimensional
subspace. Then, the function $r:F^\perp\ra \RR$ defined by
$$r( x) = \mathrm{Vol}_{k}(L\cap(F+x))$$ is $\frac{1}{k}$-concave
on its support.
\end{thm}

The crux to the proof of our main result is the following improved moments inequalities whose proof is postponed to Section~\ref{sec:prop}.
\begin{prop}\label{prop:moment}
Let $k \ge 1$ be an integer and $p$ be a $\tfrac{1}{k-1}$-concave distribution on $\RR$. It holds that
\begin{equation}\label{ineq:moment23}
 \left| \mu_3^3(p)\right|  \le 2 \sqrt{\frac{k+2}{k}} \frac{k-1}{k+3} \mu_2^3 (p) .
\end{equation}
Suppose furthermore that $0 \in \mathrm{Supp}(p)$. Then, we have that
\begin{equation}\label{ineq:moment12}
\mu_1^2(p) \le k(k+2) \mu_2^2 (p).
\end{equation}
\end{prop}
\begin{rmk}
As we will see in the proof of Proposition~\ref{prop:moment}, inequalities~\eqref{ineq:moment23} and~\eqref{ineq:moment12} are both sharp. By assuming $p$ to be centered (\ie, $\mu_1(p) = 0$) and letting $k\ra +\infty$, inequality~\eqref{ineq:moment23} recovers \cite[Lemma 2]{bubeck2015entropic}. Also, the condition that $0\in \mathrm{Supp}(p)$ for inequality~\eqref{ineq:moment12} is necessary. This can be seen by substituting, for example, $p = \delta_{t}$ for any $t \neq 0$ into \eqref{ineq:moment12}.
\end{rmk}

\section{Self-Concordance of the Universal Barrier}\label{sec:uni-bar}
Now we have enough tools at our disposal to prove the main result of this paper.
\begin{thm}\label{thm:main}
For any $n \ge 1 $ and proper convex domain $K\subseteq \RR^n$, the universal barrier $\phi$ is an $n$-self-concordant barrier for $K$.
\end{thm}

\begin{proof}
That $\phi$ is a barrier on $K$ follows from \cite[Theorem 2.5.1]{nesterov1994interior}. It remains to show that $\phi$ satisfies the differential inequalities~\eqref{ineq:sc} and~\eqref{ineq:nu_sc} with $\nu = n$.

Let any $x\in \mathrm{int}(K)$ and $h\in\RR^n$ be given. Then $K^\circ (x)$ is a convex body containing the origin. Also, let $p$ be the marginal distribution of $K^\circ (x)$ along $h$, \ie, $p = p\left(K^\circ (x),h;\cdot\right)$. Since $ K^\circ (x)$ contains the origin, we have that $\mathrm{Supp}(p)$ is a non-degenerate closed interval and $0 \in\mathrm{Supp}(p)$. 
Furthermore, Theorem~\ref{thm:brunn} shows that $p$ is a $\tfrac{1}{n-1}$-concave distribution on $\RR$.
Hence, by Proposition~\ref{prop:moment}, we have that
\begin{equation}\label{ineq:main_pf_2}
\mu_3^3 \le 2 \sqrt{\frac{n+2}{n}} \frac{n-1}{n+3} \mu_2^3 
\end{equation}
and that
\begin{equation}\label{ineq:main_pf_1}
\mu_1^2 \le n(n+2) \mu_2^2 .
\end{equation}
Here we write $\mu_i$ instead of $\mu_i(p)$ for $i= 1,2,3$. 
Using Lemma~\ref{lem:direc_deri} and inequality~\eqref{ineq:main_pf_1}, we have
\begin{align*}
\frac{ \left| D \phi (x) [h]\right| }{\sqrt{ D^2 \phi(x)[h,h] }} & \le \frac{ \left| (n+1) \mu_1  \right| }{ \sqrt{(n+1)(n+2)\mu_2^2  + (n+1) \mu_1^2 } } \\
& \le \frac{(n+1) \left| \mu_1 \right| }{ \sqrt{(n+1)(n+2)\frac{\mu_1^2}{n(n+2)} + (n+1) \mu_1^2 } } \\
& = \sqrt{n}.
\end{align*}
This shows that $\phi$ satisfies inequality~\eqref{ineq:nu_sc} with $\nu =n$.

Finally, we prove that $\phi$ satisfies inequality~\eqref{ineq:sc}. Towards that end, we first observe that $\mu_2 > 0$, for otherwise it would contradict to the non-degeneracy of $\mathrm{Supp}(p)$. Therefore, $$ D^2 \phi(x) [h,h] = (n+1) \left( (n+2)\mu_2^2  + \mu_1^2 \right) > 0. $$
Using Lemma~\ref{lem:direc_deri} and inequality~\eqref{ineq:main_pf_2}, we have
\begin{align}
\frac{  \left| D^3 \phi(x) [h,h,h] \right| }{ \left(D^2 \phi(x) [h,h] \right)^{\frac{3}{2}} } & = \frac{  \left| (n+2)(n+3) \mu_3^3  + 6(n+2)\mu_2^2 \mu_1 + 2\mu_1^3 \right| }{ \sqrt{n+1} \left( (n+2)\mu_2^2  + \mu_1^2  \right)^{\frac{3}{2}}} \notag\\
& \le  \frac{ (n+2)(n+3) \left( 2\sqrt{\frac{n+2}{n}} \frac{n-1}{n+3} \mu_2^3  \right) + 6(n+2)\mu_2^2 |\mu_1| + 2|\mu_1^3|  }{ \sqrt{n+1} \left( (n+2)\mu_2^2  + \mu_1^2 \right)^{\frac{3}{2}}} \notag\\
& = \frac{1}{ \sqrt{n+1}}\frac{  \frac{2(n-1)}{\sqrt{n}} + 6\tau + 2\tau^3  }{ \left( 1  + \tau^2 \right)^{\frac{3}{2}}},\label{ineq:main_pf_3}
\end{align}
where we set $\tau = \tfrac{| \mu_1| }{\mu_2 \sqrt{n+2}}$. Let $c_n = \frac{(n-1)}{2\sqrt{n}}$ and $\ell: \RR \ra \RR$ be the function defined by, for any $t\in \RR$,
\begin{equation*}
\ell(t) = \frac{   2t^3 + 6t + 4 c_n  }{ \left( 1  + t^2 \right)^{\frac{3}{2}}}.
\end{equation*}
Then,
\begin{equation*}
\ell'(t) = \frac{  6(- t^2 - 2c_n t +1)  }{ \left( 1  + t^2 \right)^{\frac{5}{2}}}.
\end{equation*}
The stationary points are $t =  - c_n \pm \sqrt{c_n^2 +1} =  \tfrac{1}{\sqrt{n}}\text{ or } -\sqrt{n} $. Hence,
\begin{align}
\ell(t) & \le \max\left\{ \lim_{t\ra -\infty}\ell (t) ,\ell\left( \tfrac{1}{\sqrt{n}}\right), \ell\left( -\sqrt{n}\right), \lim_{t\ra \infty}\ell (t) \right\} \notag\\
& = \max\left\{2, 2\sqrt{n+1}, -2\sqrt{\frac{n+1}{n}} , -2 \right\} \notag\\
& = 2\sqrt{n+1}.\label{ineq:main_pf_4}
\end{align}
Combining inequalities~\eqref{ineq:main_pf_3} and~\eqref{ineq:main_pf_4}, we get
\begin{equation*}
\frac{  D^3 \phi(x) [h,h,h]  }{ \left(D^2 \phi(x) [h,h] \right)^{\frac{3}{2}} } \le \frac{1}{ \sqrt{n+1}}\cdot 2\sqrt{n+1} = 2.
\end{equation*}
This completes the proof.
\end{proof}

\begin{rmk}
The upper bound $n$ on the self-concordance parameter is tight for any barrier, not just the universal barrier. It is attained by any proper convex domain containing a vertex that belongs to $n$ of the $(n-1)$-dimensional facets defined by linearly independent normals \cite[Proposition 2.3.6]{nesterov1994interior}. Proper convex domains satisfying this property include the $n$-dimensional simplex and hypercubes.
\end{rmk}

\section{Proof of Proposition~\ref{prop:moment}}\label{sec:prop}
The goal of this section is to prove Proposition~\ref{prop:moment}. We first handle the case $k = 1$, \ie, $p$ is $\infty$-concave. We claim that $S:= \{t\in \RR: p(t)>0\}$ is convex. We argue this by contradiction. Suppose $S$ is non-convex. Then there exist $t_1,t_2\in S$ and $\lambda \in (0,1)$ such that $\lambda t_1 + (1-\lambda)t_2 \not\in S$, which implies the contradiction that $0 = p(\lambda t_1 + (1-\lambda)t_2) \ge \max\{ p(t_1), p(t_2) \} > 0$. Next, we claim that $p$ is constant on $S$. Again we prove this by contradiction. Suppose there exist $t_1,t_2 \in S$ such that $p(t_1) > p(t_2)$. Then,
\begin{equation*}
p(t_2) = \lim_{\lambda \ra 0} p(\lambda t_1 + (1-\lambda)t_2) \ge \lim_{\lambda \ra 0} p(t_1) = p(t_1) > p(t_2),
\end{equation*}
which is a contradiction. So $p$ is either a uniform distribution on an interval or a Dirac delta distribution. Inequalities~\eqref{ineq:moment23} and~\eqref{ineq:moment12} are evident in both possibilities.

It remains to prove Proposition~\ref{prop:moment} for $k\ge 2$. We will first prove inequality~\eqref{ineq:moment12} in Section~\ref{sec:pf_moment12} and then inequality~\eqref{ineq:moment23} in Section~\ref{sec:pf_moment23}. Before doing that, let us provide a brief overview of the proofs. 
Each of the proofs start with a sequence of reductions and approximations. This is to modify the distribution class and turn the inequality into an equivalent variational formulation so that we can apply the following \emph{localization lemma}\footnote{Note that our notation $s$ is the $\gamma$ in the paper~\cite{fradelizi2004extreme}.}:
\begin{thm}[Localization Lemma {\cite[Theorem 2]{fradelizi2004extreme}}]\label{thm:local}
Let $m\ge 1$, $H\subseteq \RR^m$ be a compact convex set, $s\in[-1,1]$
and $f:H\rightarrow\mathbb{R}$ an upper semi-continuous function.
Also, let $\mathcal{M}(H)$ be the set of measures with support contained in $H$ and $\Pi:\mathcal{M}(H)\rightarrow\mathbb{R}$ be a convex upper semi-continuous
function. Consider the problem
\begin{equation*} 
\begin{array}{c@{\quad}l}
\hfill\displaystyle\sup_{\varphi} & \Pi(\varphi) \\
\noalign{\smallskip}
\mbox{subject to} & \varphi \text{ is $s$-concave and supported on } H,\\
& \int fd\varphi\geq0.
\end{array}
\end{equation*}
Then, the optimal value of the above problem is achieved
at either a Dirac delta distribution $\delta_{u}$ for some $u\in H$ such that $f(u)\geq0$ or a measure with density $q$ such that
\begin{enumerate}[(i)]
\item $\mbox{Supp}(q)$ is an interval $[a,b]:= \{  a + \lambda (b -a): \lambda\in [0,1] \}\subseteq H$ for some $a, b\in H$,
\item $q^s$ (or $\log q$ if $s=0$) is affine on $\mbox{Supp}(q)$,
\item $\int_a^b f(u) q(u) du=0$, and
\item $\int_{a}^{t}f(u) q(u)du>0$ for all $t\in(a,b)$ or $\int_{t}^{b}f(u) q(u) du>0$
for all $t\in(a,b)$.
\end{enumerate}
\end{thm}

\noindent For both the proofs of inequalities~\eqref{ineq:moment12} and \eqref{ineq:moment23}, after reductions, we will apply the above localization lemma to an optimization problem over probability distributions on a 1-dimensional set $H$ (\ie, $m=1$). Therefore, the quantities $a,b,t,u$ in the localization lemma are actually real-valued in our case. The localization lemma will allow us to restrict our attention to $\tfrac{1}{k-1}$-affine distributions on $\RR$, \ie, distributions of the form
\begin{equation}\label{eq:affine}
 \frac{ (\alpha t + \beta)^{k-1}\cdot \mathds{1}_{[a,b]}(t) }{\int_a^b (\alpha u + \beta)^{k-1} du},
\end{equation}
where $a \le b$ and $\alpha t + \beta \ge 0$ for any $t\in [a,b]$.
Substituting an arbitrary $\tfrac{1}{k-1}$-affine distribution into the desired inequality, the task is further reduced to proving an algebraic inequality. Finally, the proof is completed by proving the algebraic inequality using simple calculus. 

\subsection{Proof of Inequality~\eqref{ineq:moment12}}\label{sec:pf_moment12}
Let $\mathcal{P}$ be the set of $\tfrac{1}{k-1}$-concave distributions and $\bar{\mathcal{P}} \subseteq \mathcal{P}$ be the subset of distributions $p\in \mathcal{P}$ with $0 \in \mathrm{Supp}(p)$. Also, the set of $\tfrac{1}{k-1}$-affine distributions on $\RR$ (\ie,~\eqref{eq:affine}) will be denoted by $\mathcal{Q}$. We note that inequality~\eqref{ineq:moment12} is equivalent to
\begin{equation}\label{ineq:moment12_eq}
\left(\int_{-\infty}^\infty t p(t) dt \right)^2 \le \frac{k(k+2)}{(k+1)^2} \int_{-\infty}^\infty t^2 p(t) dt.
\end{equation}
So we will prove inequality~\eqref{ineq:moment12_eq} instead.

\subsubsection{To Distributions with Bounded Support}\label{sec:bdd_supp}
We first show that it suffices to prove inequality~\eqref{ineq:moment12_eq} for $p\in \bar{\mathcal{P}}$ with bounded support. This can be done by limiting arguments: Let any $p \in \bar{\mathcal{P}}$ and $\epsilon > 0$ be given. By continuity, there exists a real number $M >0$ such that 
\begin{equation}\label{ineq:prop_pf_1}
\left| \left(\int_{-\infty}^\infty t p(t) dt \right)^2 - \left( \frac{\int_{-M}^{M} t p(t) dt}{ \int_{-M}^{M} p(u)du } \right)^2 \right| \le \frac{\epsilon}{2} ,
\end{equation}
and
\begin{equation}\label{ineq:prop_pf_2}
\frac{k(k+2)}{(k+1)^2} \left| \int_{-\infty}^\infty t^2 p(t) dt - \frac{\int_{-M}^{M} t^2 p(t) dt}{ \int_{-M}^{M} p(u)du } \right| \le \frac{\epsilon}{2} .
\end{equation}
Note that the distribution $$\frac{p(t)\mathds{1}_{[-M,M]}}{\int_{-M}^{M} p(u)du} \in \bar{\mathcal{P}}$$
has a bounded support. Therefore, if inequality~\eqref{ineq:moment12_eq} holds for any distribution in $\bar{\mathcal{P}}$ with bounded support, then from inequalities~\eqref{ineq:prop_pf_1} and \eqref{ineq:prop_pf_2}, we have
\begin{equation*}
\left(\int_{-\infty}^\infty t p(t) dt \right)^2 \le \frac{k(k+2)}{(k+1)^2} \int_{-\infty}^\infty t^2 p(t) dt + \epsilon.
\end{equation*}
Since the above inequality holds for any small $\epsilon>0$, inequality~\eqref{ineq:moment12_eq} follows by taking limiting $\epsilon \searrow 0$.

\subsubsection{To Distributions with Non-negative Support}
Inequality~\eqref{ineq:moment12_eq} is equivalent to
\begin{equation*}
\Phi\left(p \right):=\frac{\left(\int_{-\infty}^\infty t p(t) dt \right)^2}{\int_{-\infty}^\infty t^2 p(t) dt} \le \frac{k(k+2)}{(k+1)^2}.
\end{equation*}
Since $\Phi$ is unchanged if we flip the distribution $p$ horizontally about $t=0$, we can assume without loss that the mean $\mu_1 (p)$ is non-negative.
By the above reduction, we could also assume that $\mathrm{Supp}(p) = [M_1, M_2]$ for some $M_1,M_2\in \RR$ with $M_1 < M_2$. Let $p_u(t) = p(t-u)$ be the distribution obtained by shifting $p$ to the right by $u$ units. Then, for any $u\ge 0$, 
\begin{align*}
& \, \frac{d \Phi\left( p_u \right) }{du} = \frac{d\ }{du}\frac{\left( \int_{M_1 + u}^{M_2+u} tp_u(t)dt\right)^2}{\int_{M_1 + u}^{M_2+u} t^2p_u(t)dt} \\
= &\, \frac{\left( \int_{M_1 + u}^{M_2+u} t^2 p_u(t)dt \right) \frac{d\ }{du} \left( \int_{M_1 }^{M_2 } (t+u) p(t) dt\right)^2  }{\left( \int_{M_1 + u}^{M_2+u} t^2 p_u(t)dt \right)^2} 
 - \frac{\left( \int_{M_1 + u}^{M_2+u} tp_u(t)dt\right)^2 \frac{d\ }{du} \left( \int_{M_1 }^{M_2 } (t+u)^2 p(t) dt\right)}{\left( \int_{M_1 + u}^{M_2+u} t^2 p_u(t)dt \right)^2}\\
= &\, \frac{ 2 \left( \int_{M_1 + u}^{M_2+u} t^2 p_u(t)dt \right) \left( \int_{M_1 }^{M_2 } (t+u) p(t) dt\right)  }{\left( \int_{M_1 + u}^{M_2+u} t^2 p_u(t)dt \right)^2}
 - \frac{2 \left( \int_{M_1 + u}^{M_2+u} tp_u(t)dt\right)^2 \left( \int_{M_1 }^{M_2 } (t+u) p(t) dt \right)}{\left( \int_{M_1 + u}^{M_2+u} t^2 p_u(t)dt \right)^2}\\
= &\, \frac{2 \mu_1(p_u)}{\left( \int_{M_1 + u}^{M_2+u} t^2 p_u(t) dt \right)^2}\left( \int_{M_1 + u}^{M_2+u} t^2 p_u(t) dt - \left( \int_{M_1 + u}^{M_2+u} t p_u(t) dt \right)^2 \right)
\ge 0,
\end{align*}
where the last inequality follows from that $\mu_1(p_u) = \mu_1(p) + u \ge 0$. This shows that shifting $p$ rightwards can only (monotonically) increase the value of $\Phi$. Therefore, we can assume that $\mathrm{Supp}(p) = [0,M_3]$ for some $M_3>0$.
\subsubsection{To Distributions with $p(0)>0$}
Here we show that it suffices to focus on distributions $p\in \bar{\mathcal{P}}$ with $p(0) > 0$.
From the above reductions, we can focus on $p\in \bar{\mathcal{P}}$ such that $\mathrm{Supp}(p) = [0,M_3]$ for some $M_3>0$. Let $\epsilon \in (0,M_3)$. By definition, we have $p(\epsilon)>0$. Consider the distribution $p_{\text{-}\epsilon}$ obtained by shifting $p$ to the left by $\epsilon$ units. We can bound the changes in the integrals in~\eqref{ineq:moment12_eq} as follows:
\begin{equation}\label{ineq:prop_pf_3}
\begin{split}
&\,\left| \int_0^{M_3} t p(t) dt - \int_{-\epsilon}^{M_3 - \epsilon} t p_{\text{-} \epsilon}(t)dt \right| 
= \left| \int_0^{M_3} t p(t) dt - \int_{0}^{M_3 } (u -\epsilon) p(u)du \right| = \epsilon,
\end{split}
\end{equation}
and
\begin{equation}\label{ineq:prop_pf_4}
\begin{split}
&\, \left| \int_0^{M_3} t^2 p(t) dt - \int_{-\epsilon}^{M_3 - \epsilon} t^2 p_{\text{-} \epsilon}(t)dt \right| \\
=&\, \left| \int_0^{M_3} t^2 p(t) dt - \int_{0}^{M_3 } (u -\epsilon)^2 p(u)du \right| = \left| 2\epsilon \int_{0}^{M_3 } u p(u)du - \epsilon^2 \right| = O(\epsilon).
\end{split}
\end{equation}
Although $p_{\text{-}\epsilon}(0) = p(\epsilon) > 0$, the support $\mathrm{Supp}(p_{\text{-}\epsilon})$ of $p_{\text{-}\epsilon}$ is not non-negative. To remedy this, we consider the truncated distribution
\begin{equation*}
p^\epsilon (t) = \frac{p(t+\epsilon)\mathds{1}_{[0,M_3 - \epsilon]}}{\int_{0}^{M_3-\epsilon} p(u)du} .
\end{equation*}
The distribution $p^\epsilon \in \bar{\mathcal{P}}$ retains all the desirable properties: $p^\epsilon(0) > 0$ and has a non-negative bounded support. Furthermore, combining inequalities~\eqref{ineq:prop_pf_3} and~\eqref{ineq:prop_pf_4} with arguments similar to those in Section~\ref{sec:bdd_supp}, we can easily show that $\Phi(p^\epsilon ) \ra \Phi(p)$ as $\epsilon \searrow 0$. Hence, we could assume without loss of generality that $p(0) >0$.

\subsubsection{To Distributions Supported in $[0,1]$}\label{sec:0-1}
Let $p \in \bar{\mathcal{P}}$. Due to above reductions, we may assume that $\mathrm{Supp}(p)=[0,M_3]$ for some $M_3>0$ and that $p(0) >0$. Consider the transformation $\tilde{p}(x) = M_3\cdot p(M_3x)$. One can easily check that $\tilde{p}$ is a probability distribution in $\bar{\mathcal{P}}$ with $\tilde{p}(0)>0$ and $\mathrm{Supp}(\tilde{p}) =[0,1]$. Furthermore, we have that
\begin{align*}
\Phi\left(\tilde{p} \right)& = \frac{\left(\int_{-\infty}^\infty t \tilde{p}(t) dt \right)^2}{\int_{-\infty}^\infty t^2 \tilde{p}(t) dt} 
= \frac{ M_3 \left( \int_{0}^1 t p( M_3 t) dt \right)^2}{\int_{0}^1 t^2 p( M_3 t) dt} 
= \frac{\left( \int_{0}^{M_3} u\cdot p(u) du \right)^2}{\int_{0}^{M_3} u^2\cdot p(u) du} 
= \Phi(p) .
\end{align*}
Therefore, it suffices henceforth to focus on the subset of distributions $p\in \bar{\mathcal{P}}$ with $p(0)>0$ and $\mathrm{Supp}(p)= [0,1]$. 

\subsubsection{To $\frac{1}{k-1}$-Affine Distributions}\label{sec:to-affine}
Let $\Psi: \bar{\mathcal{P}}\ra \RR$ be the function defined as $$\Psi (q) = (k+1)^2 \left(\int_{0}^1 t q(t) dt \right)^2 - k(k+2) \int_{0}^1 t^2 q(t) dt .$$
To prove inequality~\eqref{ineq:moment12_eq}, it suffices to prove that
\begin{equation}\label{ineq:Psi_neg}
\Psi(p) \le 0 .
\end{equation}
We recall that, by the above reductions, $p\in \mathcal{P}$ is a $\tfrac{1}{k-1}$-concave distribution with $0\in \mathrm{Supp}(p) \subseteq [0,1]$.
Consider the following problem parametrized by $\epsilon>0$:
\begin{equation}\label{opt:main}\tag{$P_\epsilon$}
\begin{array}{c@{\quad}l}
\Psi^\epsilon:=\displaystyle \sup_{q} & \Psi(q) \\
\noalign{\smallskip}
\mathrm{subject\ to} & \frac{1}{\epsilon} \int_0^\epsilon q(t) dt \ge \epsilon, \\
& q \in \bar{\mathcal{P}}',
\end{array}
\end{equation}
where $\bar{\mathcal{P}}'$ is the subset of distributions $q\in \mathcal{P}$ with $\mathrm{Supp}(q) \subseteq [0,1]$.
Inequality~\eqref{ineq:Psi_neg} can then be proven by showing that
\begin{equation}\label{ineq:prop_pf_5}
\Psi^\epsilon\le o_\epsilon(1),\ \text{ as }\epsilon \searrow 0.
\end{equation}
Indeed, since $p(0) > 0$, there exists an $\bar{\epsilon}>0$ such that for any $\epsilon\in (0, \bar{\epsilon})$, we have $\tfrac{1}{\epsilon}\int_0^\epsilon p(t) dt \ge \epsilon$. Hence, $p$ must be a feasible solution to problem~\eqref{opt:main} and hence $\Psi(p) \le \Psi^\epsilon \le o_\epsilon(1)$. Taking limit $\epsilon \searrow 0$ yields $\Psi(p) \le 0$, which is equivalent to inequality~\eqref{ineq:moment12_eq} by the above reductions.
Note that the optimal value $\Psi^\epsilon$ is always finite.

Towards proving \eqref{ineq:prop_pf_5}, we need the following corollary.
\begin{coro}\label{coro:1}
The supremum $\Psi^\epsilon$ of problem~\eqref{opt:main} is achieved by either
\begin{enumerate}
\item a Dirac distribution $\delta_{\bar{t}}$ for some $\bar{t}\in [0,\epsilon]$; or \label{case:1}
\item \label{case:2}a $\tfrac{1}{k-1}$-affine distribution $q^\epsilon \in \mathcal{Q}$ with $\mathrm{Supp}(q^\epsilon)\subseteq [0,1]$ and
\begin{equation}\label{eq:case2}
\frac{1}{\epsilon} \int_0^\epsilon q^\epsilon(t)dt = \epsilon .
\end{equation}
\end{enumerate}
\end{coro}
\begin{proof}
The result follows immediately by applying Theorem~\ref{thm:local} to problem~\eqref{opt:main} and taking $H = [0,1]$, $\Pi = \Psi$, $s = \tfrac{1}{k-1}$ and $f (t) = \tfrac{1}{\epsilon} \cdot \mathds{1}_{[0, \epsilon]}(t)- \epsilon $.
\end{proof}

\noindent Using Corollary~\ref{coro:1}, we can prove \eqref{ineq:prop_pf_5} by separately checking the two cases. We first consider Case~\ref{case:1}:
\begin{align*}
\Psi(\delta_{\bar{t}})=&\, (k+1)^2 \left(\int_{0}^1 t \delta_{\bar{t}}(t) dt \right)^2 - k(k+2) \int_{0}^1 t^2 \delta_{\bar{t}}(t) dt \\
= &\, (k+1)^2 \bar{t}\,^2 - k(k+2) \bar{t}\,^2 = \bar{t}\,^2 \le \epsilon^2 ,
\end{align*}
which implies inequality~\eqref{ineq:prop_pf_5}.
So it remains to prove inequality~\eqref{ineq:prop_pf_5} for Case~\ref{case:2}: the $\tfrac{1}{k-1}$-affine distribution $q^\epsilon\in \mathcal{Q}$.

\subsubsection{To an Algebraic Inequality}
Since $q^\epsilon\in \mathcal{Q}$ and $\mathrm{Supp}(q^\epsilon)\subseteq [0,1]$, there exist constants $a,b\in [0,1]$ and $\alpha,\beta \in\mathbb{R}$ such that $a<b$, $\alpha t + \beta \ge 0$ on $[a,b]$ and
\begin{equation*}
q^\epsilon (t) = \frac{(\alpha t + \beta)^{k-1}\cdot \mathds{1}_{[a,b]}(t)}{ \int_a^b (\alpha u + \beta)^{k-1} du }.
\end{equation*}
We claim that without loss of generality, we can assume that $a=0$ and $\beta\ge 0$. To prove the claim, from the equality~\eqref{eq:case2}, we get
\begin{equation*}
\int_{[0,\epsilon]\cap [a,b]} (\alpha t + \beta)^{k-1}dt = \epsilon^2 \int_a^b (\alpha u + \beta)^{k-1}du,
\end{equation*}
which shows that $\epsilon \in (a,b)$. Consider the shifted distribution $q^\epsilon_{\text{-}a}$. Following the same arguments for deriving~\eqref{ineq:prop_pf_3} and~\eqref{ineq:prop_pf_4}, we can show that $$|\Psi(q^\epsilon) - \Psi(q^\epsilon_{\text{-}a})| = O(\epsilon). $$
Therefore, we can assume without loss that $a = 0$ when proving \eqref{ineq:prop_pf_5} for $q^\epsilon$, which in turn implies $\beta \ge 0$.

Next, we claim that without loss of generality, we can also assume that $\alpha > 0$. Suppose $\alpha = 0$. In such a case, $q^\epsilon$ is a uniform distribution, \ie, $q^\epsilon (t) = \tfrac{1}{b}$ for all $ t \in [0,b]$. We therefore have
\begin{align*}
\Psi (q^\epsilon)& = \frac{(k+1)^2}{b^2} \left( \int_0^b t dt\right)^2 - \frac{k(k+2)}{b} \int_0^b t^2 dt \\
& =  \frac{(k+1)^2 b^2}{4}  - \frac{k(k+2)b^2}{3} \\
& = \frac{b^2}{12}\left( -k^2 - 2k + 3 \right) \le 0.
\end{align*}

For $\alpha < 0$, we consider the distribution $\tilde{q}^\epsilon(t) = q^\epsilon(b-t)$, the distribution obtained by flipping $q^\epsilon$ horizontally about $t= \tfrac{b}{2}$. In other words, $$\tilde{q}^\epsilon (t) = \frac{( -\alpha t + \alpha b+  \beta)^{k-1}\cdot \mathds{1}_{[0,b]}(t)}{ \int_0^b (\alpha u + \beta)^{k-1} du }, $$ which is again supported on $[0,b]$ and $\tfrac{1}{k-1}$-affine. In addition, $\alpha b + \beta \ge 0$ and $-\alpha > 0$. Therefore, the claim would follow if we can prove that
\begin{equation*}
\Psi( \tilde{q}^\epsilon ) \ge \Psi(q^\epsilon),
\end{equation*}
which can be easily shown to be equivalent to 
\begin{equation}\label{ineq:pf_1}
\frac{\int_0^b t(\alpha t +\beta)^{k-1} dt }{\int_0^b (\alpha t +\beta)^{k-1} dt} \le \frac{b}{2}.
\end{equation}
Inequality~\eqref{ineq:pf_1} follows immediately from the next lemma.
\begin{lem}
Let $q$ be a non-increasing distribution supported on $[0,1]$. Then the mean of $q$ is at most $\tfrac{1}{2}$.
\end{lem}
\begin{proof}
For any $t\in [0,\tfrac{1}{2}]$, we $q(t) \ge q(1-t)$ and hence
\begin{equation*}
t q(t) + (1-t) q(1-t)  \le (1-t) q(t) + t q(1-t).
\end{equation*}
Integrating both sides, we get
\begin{align*}
\int_0^{\tfrac{1}{2}} t q(t) + (1-t) q(1-t) dt & \le \int_0^{\tfrac{1}{2}} (1-t) q(t) + t q(1-t) dt ,\\
\int_0^1 t q(t) dt & \le \int_0^1 (1-t) q(t) dt ,\\
\int_0^1 t q(t) dt & \le \frac{1}{2}.
\end{align*}
\end{proof}
\noindent Therefore, we can safely ignore the case of $\alpha \le 0$.

It is obvious that inequality~\eqref{ineq:prop_pf_5} is implied by 
\begin{equation}\label{ineq:pf_2}
(k+1)^2 \left( \int_0^b t q^\epsilon (t)dt \right)^2 \le k(k+2) \int_0^b t^2 q^\epsilon (t)dt .
\end{equation}
Let $\kappa =\tfrac{\beta}{b\alpha} \ge 0$. We compute the integrals
\begin{equation}\label{eq:integral1}
\begin{split}
&\, \int_0^b t q^\epsilon (t)dt 
=  \frac{ \int_0^b t (\alpha t + \beta)^{k-1} dt }{\int_0^b (\alpha t + \beta)^{k-1} dt} 
=  \frac{ b \int_0^1 t (t + \tfrac{\beta}{b\alpha} )^{k-1} dt }{\int_0^1 ( t + \tfrac{\beta}{b\alpha})^{k-1} dt} \\
&=  b\left( \frac{ \int_0^1  (t + \kappa )^k dt }{\int_0^1 ( t + \kappa )^{k-1} dt} -\kappa\right) 
=  b \left( \frac{(1+\kappa)^{k+1} - \kappa^{k+1} }{ (1+\kappa)^k - \kappa^k}  \frac{k}{k+1} - \kappa \right), 
\end{split}
\end{equation}
and
\begin{equation}\label{eq:integral2}
\begin{split}
&\,\int_0^b t^2 q^\epsilon (t)dt 
= \frac{ \int_0^b t^2 (\alpha t + \beta)^{k-1} dt }{\int_0^b (\alpha t + \beta)^{k-1} dt} 
=  \frac{ b^2 \int_0^1 t^2 (t + \tfrac{\beta}{b\alpha} )^{k-1} dt }{\int_0^1 ( t + \tfrac{\beta}{b\alpha})^{k-1} dt} \\
& =  b^2\left( \frac{  \int_0^1  (t + \kappa )^{k+1} dt }{\int_0^1 ( t + \kappa )^{k-1} dt} - \frac{ 2 \kappa \int_0^1  (t + \kappa )^{k} dt }{\int_0^1 ( t + \kappa )^{k-1} dt} + \kappa^2 \right) \\
& = b^2 \left(\frac{(1+\kappa)^{k+2} - \kappa^{k+2} }{ (1+\kappa)^k - \kappa^k}  \frac{k}{k+2} - 2\kappa \frac{(1+\kappa)^{k+1} - \kappa^{k+1} }{ (1+\kappa)^k - \kappa^k}  \frac{k}{k+1} + \kappa^2 \right).
\end{split}
\end{equation}
Substituting \eqref{eq:integral1} and \eqref{eq:integral2} into \eqref{ineq:pf_2} yields
\begin{align*}
&\, (k+1)^2 \left( \frac{(1+\kappa)^{k+1} - \kappa^{k+1} }{ (1+\kappa)^k - \kappa^k}  \frac{k}{k+1} - \kappa \right)^2  \\
\le &\,  k(k+2) \left( \frac{(1+\kappa)^{k+2} - \kappa^{k+2} }{ (1+\kappa)^k - \kappa^k}  \frac{k}{k+2} - 2\kappa \frac{(1+\kappa)^{k+1} - \kappa^{k+1} }{ (1+\kappa)^k - \kappa^k}  \frac{k}{k+1} + \kappa^2 \right).
\end{align*}
Setting $\gamma = \tfrac{1+\kappa}{\kappa} \ge 1$, it suffices to prove the following algebraic inequality
\begin{equation}\label{ineq:gamma}
\begin{split}
& (k+1)^2 \left( \frac{ \gamma^{k+1} - 1 }{ \gamma^k - 1}  \frac{k}{k+1} - 1 \right)^2 \\
\le&\,  k(k+2)\left( \frac{ \gamma^{k+2} - 1 }{ \gamma^k - 1}  \frac{k}{k+2} - 2\cdot \frac{ \gamma^{k+1} - 1 }{ \gamma^k - 1 }  \frac{k}{k+1} +1 \right).
\end{split}
\end{equation}

\subsubsection{Proving the Algebraic Inequality~\eqref{ineq:gamma}}\label{sec:gamma}
We now conclude the proof of inequality~\eqref{ineq:moment12_eq} by proving~\eqref{ineq:gamma}. First, multiplying both sides by $(k+1) (\gamma^k - 1)^2$, we see that inequality~\eqref{ineq:gamma} becomes
\begin{align*}
0 \le &\, k(\gamma^k - 1) \left( k(k+1)(\gamma^{k+2} - 1) - 2k(k+2) (\gamma^{k+1} - 1) + (k+1)(k+2)(\gamma^k - 1) \right) \\
&\,\quad - (k+1)\left( k(\gamma^{k+1} -1 ) - (k+1)( \gamma^k -1 ) \right)^2 := f_0(\gamma).
\end{align*}
The function $f_0$ can be simplified into
\begin{equation*}
f_0(\gamma ) = 2k \gamma^{2k+1} - (k+1) \gamma^{2k} - k^2(k+1) \gamma^{k+2} + 2k(k^2+k-1)\gamma^{k+1} - (k^3+k^2-2) \gamma^k + (k-1).
\end{equation*}
Observing that $f_0(1) = 0$, it suffices to show that $f_0'(\gamma) \ge 0$ for any $\gamma \ge 1$. By simple calculation,
\begin{align*}
f_0'(\gamma)  =&\,  2k(2k+1)\gamma^{2k} - 2k(k+1)\gamma^{2k-1} - k^2(k+1)(k+2)\gamma^{k+1}\\
&\, +  2k(k+1)(k^2+k-1)\gamma^k - k(k^3 + k^2 -2)\gamma^{k-1} \\
= &\, k \gamma^{k-1} f_1(\gamma),
\end{align*}
where
\begin{equation*}
f_1(\gamma) = 2(2k+1)\gamma^{k+1} - 2(k+1)\gamma^k - k(k+1)(k+2)\gamma^2 + 2(k+1)(k^2+k-1)\gamma - (k^3+k^2-2).
\end{equation*}
Since $f_1(1) = 0$, it suffices to show that $f'_1(\gamma) \ge 0$ for any $\gamma \ge 1$. Again by simple calculation,
\begin{align*}
f'_1(\gamma) & = 2(k+1)(2k+1) \gamma^k - 2k(k+1) \gamma^{k-1} - 2k(k+1)(k+2) \gamma + 2(k+1)(k^2 +k -1) \\
& = 2(k+1)f_2(\gamma),
\end{align*}
where
\begin{equation*}
f_2(\gamma) = (2k+1)\gamma^k - k\gamma^{k-1} - k(k+2)\gamma + (k^2+k -1).
\end{equation*}
Since $f_2(1) = 0$, it suffices to show that $f'_2(\gamma) \ge 0$ for any $\gamma \ge 1$. Finally,
\begin{align*}
f_2'(\gamma) & = k(2k+1)\gamma^{k-1} - k(k-1)\gamma^{k-2} - k(k+2) \\
& = k\gamma^{k-2} \left[  (2k+1)\gamma - (k-1)\right] - k(k+2)  \\
& \ge k\left[ (2k+1) - (k-1) - (k+2) \right] \\
& = 0,
\end{align*}
where the inequality follows from the fact that $\gamma \ge 1$. This shows that $f_0(\gamma) \ge 0$ for any $\gamma \ge 1$ and hence completes the proof of inequality~\eqref{ineq:moment12_eq}.

\subsection{Proof of Inequality~\eqref{ineq:moment23}}\label{sec:pf_moment23}
Inequality~\eqref{ineq:moment23} is trivial for distributions $p\in \mathcal{P}$ with $\mu_2(p) =0$. Therefore, we assume that $\mu_2 (p) >0$. We will need the following notations. For any distribution $p\in \mathcal{P}$, we let 
\begin{equation*}
\eta(p) = \frac{\mu_1(p)}{\mu_2(p)},\ \ \Xi(p) = \left|\frac{\mu_3^3(p)}{\mu_2^3(p)}\right| \quad\text{and}\quad \Xi_k = \sup_{p\in\mathcal{P}} \Xi(p).
\end{equation*}
Then, inequality~\eqref{ineq:moment23} is equivalent to $$ \Xi_k \le 2 \sqrt{\frac{k+2}{k}} \frac{k-1}{k+3}. $$
The following observation will be useful:
\begin{equation}\label{eq:pf_4}
\sup_{p\in\mathcal{P}} \Xi(p) = \sup_{p\in\mathcal{P}} \frac{\mu_3^3(p)}{\mu_2^3(p)}.
\end{equation}
To prove it, for any $p\in \mathcal{P}$ such that $$ \frac{\mu_3^3(p)}{\mu_2^3(p)} < 0, $$
we define $\tilde{p}(x) = p( - x)$. Then we have that $\tilde{p}\in \mathcal{P}$ and that
\begin{equation*}
\frac{\mu_3^3(\tilde{p})}{\mu_2^3(\tilde{p})} = - \frac{\mu_3^3(p)}{\mu_2^3(p)} >0.
\end{equation*}
\subsubsection{To $\tfrac{1}{k-1}$-Affine Distributions}
Using the formulas
\begin{equation*}
\int_{-\infty}^\infty t^2 p(t) dt = \mu_2^2(p) + \mu_1^2(p)
\end{equation*}
and
\begin{equation*}
\int_{-\infty}^\infty t^3 p(t) dt = \mu_3^3(p) +3 \mu_1(p)\mu_2^2(p) +\mu_1^3(p),
\end{equation*}
we get
\begin{equation}\label{eq:pf_5}
\begin{split}
\frac{ \int_{-\infty}^\infty t^3 p(t) dt }{\left( \int_{-\infty}^\infty t^2 p(t) dt \right)^{\frac{3}{2}}} & =  \frac{ \frac{\mu_3^3(p)}{\mu_2^3(p)} + 3\eta(p) +  \eta^3(p) }{\left( 1 + \eta^2(p) \right)^{\frac{3}{2}}}.
\end{split}
\end{equation}
Since $\mu_2(p)$ and $\mu_3(p)$ are invariant to horizontal shift of the distribution~$p$,
\begin{equation}\label{eq:pf_6}
\begin{split}
\sigma := &\,\sup_{p\in \mathcal{P}} \frac{ \int_{-\infty}^\infty t^3 p(t) dt }{\left( \int_{-\infty}^\infty t^2 p(t) dt \right)^{\frac{3}{2}}} = \sup_{\eta \in \mathbb{R}} \sup_{p\in\mathcal{P}}  \frac{\frac{\mu_3^3(p)}{\mu_2^3(p)} + 3\eta +  \eta^3 }{\left( 1 + \eta^2 \right)^{\frac{3}{2}}} \\
=&\, \sup_{\eta \in \mathbb{R}}   \frac{ \left(\sup_{p\in\mathcal{P}} \frac{\mu_3^3(p)}{\mu_2^3(p)}\right) + 3\eta +  \eta^3 }{\left( 1 + \eta^2 \right)^{\frac{3}{2}}} = \sup_{\eta \in \mathbb{R}}   \frac{ \Xi_k + 3\eta +  \eta^3 }{\left( 1 + \eta^2 \right)^{\frac{3}{2}}} ,
\end{split}
\end{equation}
where the last equality follows from \eqref{eq:pf_4}. 
This shows that we can bound $\Xi_k$ by bounding the supremum $\sigma$.
Towards that end, we approximate the supremum $\sigma$ by truncating the distribution. In particular, using similar arguments as in Section \ref{sec:bdd_supp}, one can prove that for any $\epsilon>0$, there is a real number $M>0$ such that $\sigma \leq \sigma_M + \epsilon$, where
\begin{equation}\label{opt:Q_M}
\begin{array}{c@{\quad}l}
\displaystyle \sigma_M :=\sup_{p} & \int_{-\infty}^{\infty} t^3 p(t) dt \\
\noalign{\smallskip}
\mathrm{subject\ to} & \int_{-\infty}^{\infty} t^2 p(t) dt \le 1, \\
& p \in \mathcal{P}_M.
\end{array}
\end{equation} 
and $\mathcal{P}_M\subseteq \mathcal{P}$ is the set of distributions $p\in \mathcal{P}$ with $\mathrm{Supp}(p) \subseteq [-M,M]$.
Similar to Section~\ref{sec:to-affine}, we apply Theorem~\ref{thm:local} to problem~\eqref{opt:Q_M} and obtain the following corollary.
\begin{coro}\label{coro:2}
The supremum $\sigma_M$ of problem~\eqref{opt:Q_M} is achieved by either
\begin{enumerate}
\item\label{coro:2:1} a Dirac distribution; or
\item\label{coro:2:2} a $\tfrac{1}{k-1}$-affine distribution $q\in \mathcal{Q}$.
\end{enumerate}
\end{coro}
\begin{proof}
The result follows immediately by applying Theorem~\ref{thm:local} to problem~\eqref{opt:Q_M} and taking $H = [-M,M]$, $\Pi(p) = \int_{-\infty}^{\infty} t^3 p(t) dt$, $s = \tfrac{1}{k-1}$ and $f (t) = \mathds{1}_{[-M, M]}(t) - t^2$.
\end{proof}
\noindent We can ignore Case~\ref{coro:2:1} of Corollary~\ref{coro:2} since we assumed that $\mu_2(p) >0$. Using Case~\ref{coro:2:2} of Corollary~\ref{coro:2}, we arrive at the following relation:
\begin{align*}
&\, \sup_{p\in \mathcal{P}} \frac{ \int_{-\infty}^\infty t^3 p(t) dt }{\left( \int_{-\infty}^\infty t^2 p(t) dt \right)^{\frac{3}{2}}} \leq \sigma_M + 2\epsilon 
\le \, \int_{-\infty}^\infty t^3 q(t) dt + 2\epsilon \\
\le &\,\frac{\int_{-\infty}^\infty t^3 q(t) dt }{\left( \int_{-\infty}^\infty t^2 q(t) dt \right)^{\frac{3}{2}}}+ 2\epsilon
\le \, \sup_{p\in \mathcal{Q}}\frac{\int_{-\infty}^\infty t^3 p(t) dt }{\left( \int_{-\infty}^\infty t^2 p(t) dt \right)^{\frac{3}{2}}}+ 2\epsilon.
\end{align*}
Taking limiting $\epsilon \searrow 0$ yields
\begin{equation}\label{eq:pf_7}
\sup_{p\in \mathcal{P}} \frac{ \int_{-\infty}^\infty t^3 p(t) dt }{\left( \int_{-\infty}^\infty t^2 p(t) dt \right)^{\frac{3}{2}}} = \sup_{p\in \mathcal{Q}}\frac{\int_{-\infty}^\infty t^3 p(t) dt }{\left( \int_{-\infty}^\infty t^2 p(t) dt \right)^{\frac{3}{2}}}.
\end{equation}
Combining \eqref{eq:pf_6} and \eqref{eq:pf_7} gives that
\begin{equation}\label{eq:pf_75}
\sup_{\eta \in \mathbb{R}}   \frac{ \Xi_k + 3\eta +  \eta^3 }{\left( 1 + \eta^2 \right)^{\frac{3}{2}}} = \sup_{p\in \mathcal{Q}}\frac{\int_{-\infty}^\infty t^3 p(t) dt }{\left( \int_{-\infty}^\infty t^2 p(t) dt \right)^{\frac{3}{2}}} := \sigma'.
\end{equation}

To bound $\Xi_{k}$, we consider two cases.
Case 1: $\sigma' \le \sqrt{2}$.
Putting $\eta=1$ in \eqref{eq:pf_75} gives $$\frac{\Xi_{k}+4}{2^{\frac{3}{2}}}\leq\sqrt{2}, $$
which implies $\Xi_{k}\leq0$.
Case 2: $\sigma' > \sqrt{2}$.
Let $\bar{q}\in\mathcal{Q}$ be an $\epsilon$-approximate maximizer of $\sigma'$ with $\epsilon<0.01$.
Then,
\begin{equation}\label{eq:pf_9}
\begin{split}
 \frac{ \Xi_k + 3\eta(\bar{q}) +  \eta^3(\bar{q}) }{\left( 1 + \eta^2(\bar{q}) \right)^{\frac{3}{2}}} \le &\, \sup_{\eta \in \mathbb{R}}   \frac{ \Xi_k + 3\eta +  \eta^3 }{\left( 1 + \eta^2 \right)^{\frac{3}{2}}}
= \, \sup_{p\in \mathcal{P}} \frac{ \int_{-\infty}^\infty t^3 p(t) dt }{\left( \int_{-\infty}^\infty t^2 p(t) dt \right)^{\frac{3}{2}}} \\
\le &\, \frac{ \int_{-\infty}^\infty t^3 \bar{q}(t) dt }{\left( \int_{-\infty}^\infty t^2 \bar{q}(t) dt \right)^{\frac{3}{2}}} + \epsilon
\leq \, \frac{ \Xi(\bar{q}) + 3\eta(\bar{q}) +  \eta^3(\bar{q}) }{\left( 1 + \eta^2(\bar{q}) \right)^{\frac{3}{2}}} + \epsilon,
\end{split}
\end{equation}
where the equality follows from~\eqref{eq:pf_6}, the second inequality from~\eqref{eq:pf_7} and the last inequality from~\eqref{eq:pf_5}.
On the other hand, we have
\begin{equation}\label{eq:pf_10}
1.4< \sqrt{2} - 0.01< \sigma' - \epsilon \leq\frac{\int_{-\infty}^{\infty}t^{3}\bar{q}(t)dt}{\left(\int_{-\infty}^{\infty}t^{2}\bar{q}(t)dt\right)^{\frac{3}{2}}}\leq\frac{\Xi(\bar{q})+3\eta(\bar{q})+\eta^{3}(\bar{q})}{(1+\eta^{2}(\bar{q}))^{\frac{3}{2}}},
\end{equation}
where the second inequality follows from the fact that $\sigma' > \sqrt{2}$ and $\epsilon < 0.01$, the third inequality from the fact that $\bar{q}$ is an $\epsilon$-approximate maximizer and the last inequality from~\eqref{eq:pf_5}.
Using AM-GM inequality and then inequality~\eqref{eq:pf_10}, we have
\begin{equation*}
1.4\cdot \eta^3 (\bar{q}) \le 1.4 \cdot \left( \eta^2(\bar{q} ) \right)^{\frac{3}{2}} \le 1.4 \cdot \left( \frac{1+\eta^2(\bar{q})}{2} \right)^{\frac{3}{2}} \le \Xi(\bar{q})+3\eta(\bar{q})+\eta^{3}(\bar{q}) .
\end{equation*}
Using this inequality and the Young's inequality, we obtain
\[
\eta^{3}(\bar{q})\leq  \frac{1}{0.4} \left( \Xi(\bar{q})+3\cdot \eta(\bar{q}) \right) \leq \frac{1}{0.4} \left( \Xi(\bar{q})+\frac{\eta(\bar{q})^{3}}{3} + \frac{2\cdot 3^{\frac{3}{2}}}{3} \right) \le \frac{5}{2} \Xi(\bar{q}) + \frac{5}{6} \eta(\bar{q})^{3} + \frac{5}{3} 3^{\frac{3}{2}},
\]
which implies that 
\begin{equation}\label{ineq:pf_3}
\eta^{3}(\bar{q}) \leq 15\cdot \Xi(\bar{q})+ 10\cdot 3^{\frac{3}{2}}.
\end{equation}
We shall use the following elementary inequality: for any $r\ge 1$ and $\omega > 0$, there exists a constant $C_{r,\omega}>0$ such that
\begin{equation}\label{ineq:pf_4}
(y_1 + y_2 )^r \le (1+ \omega) y_1^r + C_{r,\omega} y_2^r.
\end{equation}
By inequalities~\eqref{eq:pf_9}, \eqref{ineq:pf_4} and \eqref{ineq:pf_3}, for any $\omega > 0$, there exists some $C_\omega > 0$ such that $$\Xi_k \le \Xi(\bar{q}) + \epsilon \cdot \left( 1 + \eta^2(\bar{q}) \right)^{\frac{3}{2}} \le \Xi(\bar{q}) + \epsilon \cdot (1+ \omega) \eta^{3}(\bar{q}) + \epsilon \cdot C_\omega \le \left(1 + O( \epsilon) \right) \Xi(\bar{q}) + O(\epsilon) + \epsilon \cdot C_\omega . $$
Since $\epsilon$ and $\omega$ are arbitrarily small and we can decrease $\epsilon$ and $\omega$ at rates such that $\epsilon \cdot C_\omega \rightarrow 0$ as $\epsilon,\omega \searrow 0$, it suffices to show that 
\begin{equation}\label{ineq:affine_main}
\Xi(\bar{q})  \le 2 \sqrt{\frac{k+2}{k}} \frac{k-1}{k+3}.
\end{equation}

\subsubsection{To an Algebraic Inequality}
Now we prove inequality~\eqref{ineq:affine_main}. The case of $\alpha =0$ or $a=b$ is trivial. So we assume that $\alpha \neq 0$ and $a<b$. 
We state without proof the simple observation that $\Xi$ is invariant under translation and scaling.
\begin{lem}\label{lem:1}
Let $p\in\mathcal{Q}$ and $\tilde{p}(t) = |\tilde{\alpha}|\cdot  p(\tilde{\alpha}t + \tilde{\beta})$, where $\tilde{\alpha},\tilde{\beta}\in\mathbb{R}$ are real numbers with $\tilde{\alpha} \neq 0$. Then $\tilde{p}\in \mathcal{Q}$ and $\Xi(p) = \Xi(\tilde{p})$.
\end{lem}
\noindent By Lemma~\ref{lem:1}, instead of  $\bar{q}$, it suffices to consider the distribution
\begin{equation*}
\tilde{q}(t) = \frac{  t^{k-1}\cdot \mathds{1}_{[\alpha a + \beta, \alpha b + \beta]}(t) }{\int_{\alpha a + \beta}^{\alpha b + \beta}  u ^{k-1} du}.
\end{equation*}
Case 1: $\alpha a + \beta = 0$. Since $\alpha b + \beta >\alpha a + \beta= 0$, Lemma~\ref{lem:1} allows us to simplify $\tilde{q}$ further to
\begin{equation*}
\hat{q}(t) = \frac{  t^{k-1}\cdot \mathds{1}_{[0, 1]}(t) }{\int_{0}^{1}  u ^{k-1} du} = k\cdot t^{k-1}\cdot \mathds{1}_{[0, 1]}(t).
\end{equation*}
We compute
\begin{align*}
\mu_2^2(\hat{q}) & = \int_0^1 t^2\hat{q}(t)dt - \left(\int_0^1 t\hat{q}(t)dt\right)^2 
= \frac{k}{k+2} - \left( \frac{k}{k+1}\right)^2 = \frac{k}{(k+1)^2(k+2)},
\end{align*}
and
\begin{align*}
\mu_3^3(\hat{q}) & = \int_0^1 t^3 \hat{q}(t)dt - 3\mu_1(\hat{q})\mu_2^2(\hat{q}) - \mu_1^3(\hat{q}) \\
& = \frac{k}{k+3} - 3\cdot \frac{k}{k+1} \cdot  \frac{k}{(k+1)^2(k+2)}  - \left( \frac{k}{k+1}\right)^3 \\
& = \frac{-2k(k-1)}{(k+1)^3(k+2)(k+3)}.
\end{align*}
Therefore, 
\begin{align*}
\Xi(\bar{q})  = \Xi(\hat{q})  = \frac{2k(k-1)}{(k+1)^3(k+2)(k+3)}\cdot \frac{(k+1)^3(k+2)^{\frac{3}{2}}}{k^{\frac{3}{2}}} = 2 \sqrt{\frac{k+2}{k}} \frac{k-1}{k+3}.
\end{align*}
Case 2: $\alpha a + \beta > 0$. Again using Lemma~\ref{lem:1}, instead of $\tilde{q}$, it suffices to consider
\begin{equation*}
\check{q}(t) = \frac{  t^{k-1}\cdot \mathds{1}_{[1,\gamma]}(t) }{\int_1^\gamma  u ^{k-1} du},
\end{equation*}
where $\gamma = \tfrac{\alpha b + \beta}{\alpha a + \beta} > 1$. Then it suffices to show that
\begin{align*}
& \,\frac{4(k+2)(k-1)^2}{k(k+3)^2} \ge \left(\Xi(\check{q})\right)^2 \\
=&\, \frac{\left( \int_1^{\gamma} t^3 \check{q}(t) dt - 3 \left( \int_1^{\gamma} t \check{q}(t) dt\right)\left( \int_1^{\gamma} t^2 \check{q}(t) dt\right) + 2 \left(\int_1^{\gamma} t \check{q}(t) dt\right)^3 \right)^2}{\left( \int_1^{\gamma} t^2 \check{q}(t) dt - \left(\int_1^{\gamma} t \check{q}(t) dt\right)^2 \right)^3} \\
=&\, \frac{\left(  \left(\int_1^{\gamma} t^{k-1} dt\right)^2\left(\int_1^{\gamma} t^{k+2} dt\right) - 3\left(\int_1^{\gamma} t^{k-1} dt\right)\left(\int_1^{\gamma} t^{k} dt\right) \left(\int_1^{\gamma} t^{k+1} dt\right) + 2 \left(\int_1^{\gamma} t^{k} dt\right)^3 \right)^2}{\left( \left(\int_1^{\gamma} t^{k-1} dt\right)\left(\int_1^{\gamma} t^{k+1} dt\right) - \left(\int_1^{\gamma} t^{k} dt\right)^2 \right)^3} \\
=&\, \frac{\left( \left( \frac{\gamma^k -1}{k} \right)^2 \left(\frac{\gamma^{k+3} -1}{k+3}\right) - 3\left(\frac{\gamma^{k} -1}{k}\right)\left(\frac{\gamma^{k+1} -1}{k+1}\right) \left(\frac{\gamma^{k+2} -1}{k+2}\right) + 2 \left(\frac{\gamma^{k+1} -1}{k+1}\right)^3 \right)^2}{\left( \left(\frac{\gamma^{k} -1}{k}\right)\left(\frac{\gamma^{k+2} -1}{k+2}\right) - \left(\frac{\gamma^{k+1} -1}{k+1}\right)^2 \right)^3}.
\end{align*}
Upon rearranging terms, the above inequality is equivalent to 
\begin{equation}\label{eq:pf_8}
\begin{split}
0& \le 4 (k-1)^2 \left((k+1)^2 \left(\gamma ^k-1\right) \left(\gamma ^{k+2}-1\right)-k (k+2) \left(\gamma ^{k+1}-1\right)^2\right)^3 \\
&\quad-\Bigg(2 k^2 (k+2) (k+3)\left(\gamma ^{k+1}-1\right)^3+(k+2) (k+1)^3 \left(\gamma ^k-1\right)^2 \left(\gamma ^{k+3}-1\right)\\
&\quad-3 k (k+3) (k+1)^2 \left(\gamma ^k-1\right)
   \left(\gamma ^{k+1}-1\right) \left(\gamma ^{k+2}-1\right)\Bigg)^2:= g(\gamma, k).
\end{split}
\end{equation}
The proof of inequality~\eqref{ineq:moment23} is thus completed by the following lemma.
\begin{lem}\label{ineq:gg}
For any $\gamma \ge 1$ and integer $k \ge 2$, $g(\gamma, k) \ge 0$.
\end{lem}
\noindent The proof of Lemma~\ref{ineq:gg}, which is provided in Appendix~\ref{app:A}, is straightforward and uses on only elementary calculus, despite its tediousness.


\section{Conclusion}\label{sec:conclusion}
This paper showed that the universal barrier of Nesterov and Nemirovski~\cite{nesterov1994interior} is $n$-self-concordant on any proper convex domain in $\RR^n$. The key to the proof of this result is a pair of new, sharp moment inequalities for $s$-concave distributions, which could be of independent interest. Currently, these inequalities concern only the first three moments. An interesting research question would be to generalize them to higher moments.

%

%
%

\section*{Acknowledgments.}
The authors are indebted to Manuel Kauers for kindly sharing with us his idea for proving Lemma~\ref{ineq:gg} using cylindrical algebraic decomposition (which is omitted in the final version of the paper) and thank S{\'e}bastien Bubeck and Santosh S. Vempala for helpful discussions. This work was supported in part by NSF awards CCF-1749609, CCF-1740551, DMS-1839116 and EPSRC grant EP/M027856/1.

\bibliographystyle{abbrv}
\bibliography{references}

\begin{appendix}
\section{Proof of Lemma~\ref{ineq:gg}}
\label{app:A}
The purpose of this section is to provide an elementary proof of Lemma~\ref{ineq:gg}, \ie, $g(\gamma, k) \ge 0 $ for any $\gamma \ge 1$ and integer $k \ge 2$. For simplicity, in this section, we omit the second input $k$ from the function $g$.

The idea of the proof is straightforward and uses only elementary calculus. Specifically, since $g(1) = 0$, the goal of establishing the non-negativity of $g$ on $[1,\infty)$ reduces to proving that $g' \ge 0$ on $[1,\infty)$. We then compute $g'$ and extract its non-negative factors. As it turns out, $g' (1) = 0$. Therefore, it suffices to show that $g'' \ge 0$ on $[1,\infty)$. We find that, once again, the second derivative $g''$ vanishes at $\gamma = 1$, \ie, $g''(1) = 0$. Our goal thus reduces to proving $g^{(3)} (\gamma) \ge 0$ on $[1,\infty)$. We keep applying such arguments to reduce our goal to proving the non-negativity of the next derivative until the 17th times, where we can show by simple algebra that $g^{(17)} \ge 0$ on $[1,\infty)$.

All the derivatives of $g$ are polynomials with exponents and coefficients depending on the integer $k$.
In the course of the above reductions for proving the desired inequality~\eqref{eq:pf_8}, some of the terms of the derivatives of $g$ have exponents $k-3$, $k-4$, or $k-5$, which could possibly be negative and hence break our arguments. Therefore, we separately handle the boundary cases $k =2,3,4$.

\subsection{The boundary cases $k =2,3,4$}
We first prove the inequality for the case of $2\le k \le 4$. 

For $k = 2$ and $\gamma \ge 1$, $$g(\gamma) = 108 (\gamma - 1)^{12} \gamma^2 ( 1 + 6 \gamma + \gamma^2) \ge 0 .$$
For $k = 3$ and $\gamma \ge 1$, $$g(\gamma) = 512 (\gamma - 1)^{12} \gamma^3 ( 2 + 15 \gamma + 60 \gamma^2 + 96 \gamma^3 + 60 \gamma^4 + 15 \gamma^5 + 2 \gamma^6 ) \ge 0 .$$
For $k = 4$ and $\gamma \ge 1$, $$g(\gamma) = 4500 (\gamma - 1)^{12} \gamma^4 ( 1 + 8 \gamma + 35 \gamma^2 + 110 \gamma^3 + 212 \gamma^4 + 268 \gamma^5 + 212 \gamma^6 + 110 \gamma^7 + 35 \gamma^8 + 8 \gamma^9 + \gamma^{10} ) \ge 0 .$$
Therefore, $g(\gamma) \ge 0$  for any $\gamma \ge 1$ and integer $k =2,3,4 $.

\subsection{The case of $k\ge 5$}
Now we prove the inequality for the case of $k \ge 5$. Although the proof for this case is tedious, the strategy is straightforward and uses only simple calculus.

We first factorize $g$ as $g(\gamma) = (k+1)^3 (\gamma - 1)^2 \gamma^k g_0 (\gamma)$, where
\begin{align*}
& g_0(\gamma)\\
= &  \left(4 k^2-8 k+4\right) \gamma ^{4 k+4} + \left(-4 k^2-4 k+8\right) \gamma ^{4 k+3} + \left(-k^5-7 k^4+5 k^3+15 k^2-12\right) \gamma ^{3 k+4} \\
& + \left(4 k^5+28 k^4+16 k^3-36 k^2+12 k-24\right) \gamma ^{3 k+3} + \left(-6 k^5-42 k^4-66 k^3-6 k^2+48 k\right) \gamma ^{3 k+2} \\
& + \left(4 k^5+28 k^4+64 k^3+52 k^2+4 k-8\right) \gamma ^{3 k+1} + \left(-k^5-7 k^4-19 k^3-25 k^2-16 k-4\right) \gamma ^{3 k} \\
& + \left(-6 k^5-18 k^4-18 k^3+6 k^2+24 k+12\right) \gamma ^{2 k+4} + \left(24 k^5+72 k^4+48 k^3-12 k^2-12 k+24\right) \gamma ^{2 k+3} \\
& + \left(-36 k^5-108 k^4-60 k^3+12 k^2-96 k\right) \gamma ^{2 k+2} + \left(24 k^5+72 k^4+48 k^3-12 k^2-12 k+24\right) \gamma ^{2 k+1}\\
& + \left(-6 k^5-18 k^4-18 k^3+6 k^2+24 k+12\right) \gamma ^{2 k} + \left(-k^5-7 k^4-19 k^3-25 k^2-16 k-4\right) \gamma ^{k+4} \\
& + \left(4 k^5+28 k^4+64 k^3+52 k^2+4 k-8\right) \gamma ^{k+3} + \left(-6 k^5-42 k^4-66 k^3-6 k^2+48 k\right) \gamma ^{k+2} \\
& + \left(4 k^5+28 k^4+16 k^3-36 k^2+12 k-24\right) \gamma ^{k+1} + \left(-k^5-7 k^4+5 k^3+15 k^2-12\right) \gamma ^k \\
& + \left(-4 k^2-4 k+8\right) \gamma  + 4 k^2-8 k+4 .
\end{align*}
It can be checked that $g_0 (1) = 0$. Therefore, it suffices to show that $g_0'(\gamma) \ge 0$ for $\gamma \ge 1$. The derivative of $g_0$ is given by $g_0'(\gamma) = g_1 (\gamma)$, where
\begin{align*}
& g_1(\gamma) \\
= & \left(16 k^3-16 k^2-16 k+16\right) \gamma ^{4 k+3} + \left(-16 k^3-28 k^2+20 k+24\right) \gamma ^{4 k+2} \\
& + \left(-3 k^6-25 k^5-13 k^4+65 k^3+60 k^2-36 k-48\right) \gamma ^{3 k+3} \\
& + \left(12 k^6+96 k^5+132 k^4-60 k^3-72 k^2-36 k-72\right) \gamma ^{3 k+2} \\
& + \left(-18 k^6-138 k^5-282 k^4-150 k^3+132 k^2+96 k\right) \gamma ^{3 k+1} \\
& + \left(12 k^6+88 k^5+220 k^4+220 k^3+64 k^2-20 k-8\right) \gamma ^{3 k}  + \left(-3 k^6-21 k^5-57 k^4-75 k^3-48 k^2-12 k\right) \gamma ^{3 k-1} \\
& + \left(-12 k^6-60 k^5-108 k^4-60 k^3+72 k^2+120 k+48\right) \gamma ^{2 k+3} \\
& + \left(48 k^6+216 k^5+312 k^4+120 k^3-60 k^2+12 k+72\right) \gamma ^{2 k+2} \\
& + \left(-72 k^6-288 k^5-336 k^4-96 k^3-168 k^2-192 k\right) \gamma ^{2 k+1} \\
& + \left(48 k^6+168 k^5+168 k^4+24 k^3-36 k^2+36 k+24\right) \gamma ^{2 k} \\
& + \left(-12 k^6-36 k^5-36 k^4+12 k^3+48 k^2+24 k\right) \gamma ^{2 k-1} \\
& + \left(-k^6-11 k^5-47 k^4-101 k^3-116 k^2-68 k-16\right) \gamma ^{k+3} \\
& + \left(4 k^6+40 k^5+148 k^4+244 k^3+160 k^2+4 k-24\right) \gamma ^{k+2} \\
& + \left(-6 k^6-54 k^5-150 k^4-138 k^3+36 k^2+96 k\right) \gamma ^{k+1}  + \left(4 k^6+32 k^5+44 k^4-20 k^3-24 k^2-12 k-24\right) \gamma ^k \\
& + \left(-k^6-7 k^5+5 k^4+15 k^3-12 k\right) \gamma ^{k-1}  -4 k^2-4 k+8.
\end{align*}
It can be checked that $g_1 (1) = 0$. Therefore, it suffices to show that $g_1'(\gamma) \ge 0$ for $\gamma \ge 1$. The derivative of $g_1$ is given by $g_1'(\gamma) = \gamma^{k-2} g_2 (\gamma)$, where
\begin{align*}
& g_2 (\gamma) \\
= & \left(64 k^4-16 k^3-112 k^2+16 k+48\right) \gamma ^{3 k+4} + \left(-64 k^4-144 k^3+24 k^2+136 k+48\right) \gamma ^{3 k+3} \\
& + \left(-9 k^7-84 k^6-114 k^5+156 k^4+375 k^3+72 k^2-252 k-144\right) \gamma ^{2 k+4} \\
& + \left(36 k^7+312 k^6+588 k^5+84 k^4-336 k^3-252 k^2-288 k-144\right) \gamma ^{2 k+3} \\
& + \left(-54 k^7-432 k^6-984 k^5-732 k^4+246 k^3+420 k^2+96 k\right) \gamma ^{2 k+2} \\
& + \left(36 k^7+264 k^6+660 k^5+660 k^4+192 k^3-60 k^2-24 k\right) \gamma ^{2 k+1} \\
& + \left(-9 k^7-60 k^6-150 k^5-168 k^4-69 k^3+12 k^2+12 k\right) \gamma ^{2 k} \\
& + \left(-24 k^7-156 k^6-396 k^5-444 k^4-36 k^3+456 k^2+456 k+144\right) \gamma ^{k+4} \\
& + \left(96 k^7+528 k^6+1056 k^5+864 k^4+120 k^3-96 k^2+168 k+144\right) \gamma ^{k+3} \\
& + \left(-144 k^7-648 k^6-960 k^5-528 k^4-432 k^3-552 k^2-192 k\right) \gamma ^{k+2} \\
& + \left(96 k^7+336 k^6+336 k^5+48 k^4-72 k^3+72 k^2+48 k\right) \gamma ^{k+1} \\
& + \left(-24 k^7-60 k^6-36 k^5+60 k^4+84 k^3-24 k\right) \gamma ^k \\
& + \left(-k^7-14 k^6-80 k^5-242 k^4-419 k^3-416 k^2-220 k-48\right) \gamma ^4 \\
& + \left(4 k^7+48 k^6+228 k^5+540 k^4+648 k^3+324 k^2-16 k-48\right) \gamma ^3 \\
& + \left(-6 k^7-60 k^6-204 k^5-288 k^4-102 k^3+132 k^2+96 k\right) \gamma ^2  \\
& + \left(4 k^7+32 k^6+44 k^5-20 k^4-24 k^3-12 k^2-24 k\right) \gamma -k^7-6 k^6+12 k^5+10 k^4-15 k^3-12 k^2+12 k .
\end{align*}
It can be checked that $g_2 (1) = 0$. Therefore, it suffices to show that $g_2'(\gamma) \ge 0$ for $\gamma \ge 1$. The derivative of $g_2$ is given by $g_2'(\gamma) = 2(k+1) g_3(\gamma)$, where
\begin{align*}
& g_3 (\gamma)  \\
= & \left(96 k^4+8 k^3-208 k^2+8 k+96\right) \gamma ^{3 k+3} + \left(-96 k^4-216 k^3+36 k^2+204 k+72\right) \gamma ^{3 k+2} \\
& + \left(-9 k^7-93 k^6-189 k^5+117 k^4+570 k^3+252 k^2-360 k-288\right) \gamma ^{2 k+3} \\
& + \left(36 k^7+330 k^6+726 k^5+240 k^4-450 k^3-306 k^2-360 k-216\right) \gamma ^{2 k+2} \\
& + \left(-54 k^7-432 k^6-984 k^5-732 k^4+246 k^3+420 k^2+96 k\right) \gamma ^{2 k+1} \\
& + \left(36 k^7+246 k^6+546 k^5+444 k^4+78 k^3-42 k^2-12 k\right) \gamma ^{2 k} \\
& + \left(-9 k^7-51 k^6-99 k^5-69 k^4+12 k^2\right) \gamma ^{2 k-1} \\
& + \left(-12 k^7-114 k^6-396 k^5-618 k^4-288 k^3+444 k^2+696 k+288\right) \gamma ^{k+3} \\
& + \left(48 k^7+360 k^6+960 k^5+1056 k^4+300 k^3-168 k^2+108 k+216\right) \gamma ^{k+2} \\
& + \left(-72 k^7-396 k^6-732 k^5-492 k^4-252 k^3-456 k^2-192 k\right) \gamma ^{k+1} \\
& + \left(48 k^7+168 k^6+168 k^5+24 k^4-36 k^3+36 k^2+24 k\right) \gamma ^k + \left(-12 k^7-18 k^6+30 k^4+12 k^3-12 k^2\right) \gamma ^{k-1} \\
& + \left(-2 k^6-26 k^5-134 k^4-350 k^3-488 k^2-344 k-96\right) \gamma ^3 \\
& + \left(6 k^6+66 k^5+276 k^4+534 k^3+438 k^2+48 k-72\right) \gamma ^2 \\
& +  \left(-6 k^6-54 k^5-150 k^4-138 k^3+36 k^2+96 k\right) \gamma + 2 k^6+14 k^5+8 k^4-18 k^3+6 k^2-12 k .
\end{align*}
It can be checked that $g_3 (1) = 0$. Therefore, it suffices to show that $g_3'(\gamma) \ge 0$ for $\gamma \ge 1$. The derivative of $g_3$ is given by $g_3'(\gamma) = 3 g_4(\gamma)$, where
\begin{align*}
& g_4 (\gamma) \\
= & \left(96 k^5+104 k^4-200 k^3-200 k^2+104 k+96\right) \gamma ^{3 k+2} \\
& + \left(-96 k^5-280 k^4-108 k^3+228 k^2+208 k+48\right) \gamma ^{3 k+1} \\
& + \left(-6 k^8-71 k^7-219 k^6-111 k^5+497 k^4+738 k^3+12 k^2-552 k-288\right) \gamma ^{2 k+2} \\
& + \left(24 k^8+244 k^7+704 k^6+644 k^5-140 k^4-504 k^3-444 k^2-384 k-144\right) \gamma ^{2 k+1} \\
& + \left(-36 k^8-306 k^7-800 k^6-816 k^5-80 k^4+362 k^3+204 k^2+32 k\right) \gamma ^{2 k} \\
& + \left(24 k^8+164 k^7+364 k^6+296 k^5+52 k^4-28 k^3-8 k^2\right) \gamma ^{2 k-1} \\
& + \left(-6 k^8-31 k^7-49 k^6-13 k^5+23 k^4+8 k^3-4 k^2\right) \gamma ^{2 k-2} \\
& + \left(-4 k^8-50 k^7-246 k^6-602 k^5-714 k^4-140 k^3+676 k^2+792 k+288\right) \gamma ^{k+2} \\
& + \left(16 k^8+152 k^7+560 k^6+992 k^5+804 k^4+144 k^3-76 k^2+144 k+144\right) \gamma ^{k+1} \\
& + \left(-24 k^8-156 k^7-376 k^6-408 k^5-248 k^4-236 k^3-216 k^2-64 k\right) \gamma ^k \\
& + \left(16 k^8+56 k^7+56 k^6+8 k^5-12 k^4+12 k^3+8 k^2\right) \gamma ^{k-1} \\
& + \left(-4 k^8-2 k^7+6 k^6+10 k^5-6 k^4-8 k^3+4 k^2\right) \gamma ^{k-2} \\
& + \left(-2 k^6-26 k^5-134 k^4-350 k^3-488 k^2-344 k-96\right) \gamma ^2 \\
& +  \left(4 k^6+44 k^5+184 k^4+356 k^3+292 k^2+32 k-48\right) \gamma -2 k^6-18 k^5-50 k^4-46 k^3+12 k^2+32 k  .
\end{align*}
It can be checked that $g_4 (1) = 0$. Therefore, it suffices to show that $g_4'(\gamma) \ge 0$ for $\gamma \ge 1$. The derivative of $g_4$ is given by $g_4'(\gamma) = 2 g_5(\gamma)$, where
\begin{align*}
& g_5 (\gamma ) \\
= & \left(144 k^6+252 k^5-196 k^4-500 k^3-44 k^2+248 k+96\right) \gamma ^{3 k+1} \\
& + \left(-144 k^6-468 k^5-302 k^4+288 k^3+426 k^2+176 k+24\right) \gamma ^{3 k} \\
& + \left(-6 k^9-77 k^8-290 k^7-330 k^6+386 k^5+1235 k^4+750 k^3-540 k^2-840 k-288\right) \gamma ^{2 k+1} \\
& + \left(24 k^9+256 k^8+826 k^7+996 k^6+182 k^5-574 k^4-696 k^3-606 k^2-336 k-72\right) \gamma ^{2 k} \\
& + \left(-36 k^9-306 k^8-800 k^7-816 k^6-80 k^5+362 k^4+204 k^3+32 k^2\right) \gamma ^{2 k-1} \\
& + \left(24 k^9+152 k^8+282 k^7+114 k^6-96 k^5-54 k^4+6 k^3+4 k^2\right) \gamma ^{2 k-2} \\
& + \left(-6 k^9-25 k^8-18 k^7+36 k^6+36 k^5-15 k^4-12 k^3+4 k^2\right) \gamma ^{2 k-3} \\
& + \left(-2 k^9-29 k^8-173 k^7-547 k^6-959 k^5-784 k^4+198 k^3+1072 k^2+936 k+288\right) \gamma ^{k+1} \\
& + \left(8 k^9+84 k^8+356 k^7+776 k^6+898 k^5+474 k^4+34 k^3+34 k^2+144 k+72\right) \gamma ^k \\
& + \left(-12 k^9-78 k^8-188 k^7-204 k^6-124 k^5-118 k^4-108 k^3-32 k^2\right) \gamma ^{k-1} \\
& + \left(8 k^9+20 k^8-24 k^6-10 k^5+12 k^4-2 k^3-4 k^2\right) \gamma ^{k-2} \\
& + \left(-2 k^9+3 k^8+5 k^7-k^6-13 k^5+2 k^4+10 k^3-4 k^2\right) \gamma ^{k-3} \\
& +  \left(-2 k^6-26 k^5-134 k^4-350 k^3-488 k^2-344 k-96\right) \gamma + 2 k^6+22 k^5+92 k^4+178 k^3+146 k^2+16 k-24 .
\end{align*}
It can be checked that $g_5 (1) = 0$. Therefore, it suffices to show that $g_5'(\gamma) \ge 0$ for $\gamma \ge 1$. The derivative of $g_5$ is given by $g_5'(\gamma) = g_6(\gamma)$, where
\begin{align*}
& g_6 (\gamma ) \\
= &\left(432 k^7+900 k^6-336 k^5-1696 k^4-632 k^3+700 k^2+536 k+96\right) \gamma ^{3 k} \\
& + \left(-432 k^7-1404 k^6-906 k^5+864 k^4+1278 k^3+528 k^2+72 k\right) \gamma ^{3 k-1}\\
& + \left(-12 k^{10}-160 k^9-657 k^8-950 k^7+442 k^6+2856 k^5+2735 k^4-330 k^3-2220 k^2-1416 k-288\right) \gamma ^{2 k}\\
& + \left(48 k^{10}+512 k^9+1652 k^8+1992 k^7+364 k^6-1148 k^5-1392 k^4-1212 k^3-672 k^2-144 k\right) \gamma ^{2 k-1}\\
& + \left(-72 k^{10}-576 k^9-1294 k^8-832 k^7+656 k^6+804 k^5+46 k^4-140 k^3-32 k^2\right) \gamma ^{2 k-2}\\
& + \left(48 k^{10}+256 k^9+260 k^8-336 k^7-420 k^6+84 k^5+120 k^4-4 k^3-8 k^2\right) \gamma ^{2 k-3}\\
& + \left(-12 k^{10}-32 k^9+39 k^8+126 k^7-36 k^6-138 k^5+21 k^4+44 k^3-12 k^2\right) \gamma ^{2 k-4}\\
& + \left(-2 k^{10}-31 k^9-202 k^8-720 k^7-1506 k^6-1743 k^5-586 k^4+1270 k^3+2008 k^2+1224 k+288\right) \gamma ^k\\
& + \left(8 k^{10}+84 k^9+356 k^8+776 k^7+898 k^6+474 k^5+34 k^4+34 k^3+144 k^2+72 k\right) \gamma ^{k-1}\\
& + \left(-12 k^{10}-66 k^9-110 k^8-16 k^7+80 k^6+6 k^5+10 k^4+76 k^3+32 k^2\right) \gamma ^{k-2}\\
& + \left(8 k^{10}+4 k^9-40 k^8-24 k^7+38 k^6+32 k^5-26 k^4+8 k^2\right) \gamma ^{k-3}\\
& + \left(-2 k^{10}+9 k^9-4 k^8-16 k^7-10 k^6+41 k^5+4 k^4-34 k^3+12 k^2\right) \gamma ^{k-4} \\
& -2 k^6-26 k^5-134 k^4-350 k^3-488 k^2-344 k-96.
\end{align*}
It can be checked that $g_6 (1) = 0$. Therefore, it suffices to show that $g_6'(\gamma) \ge 0$ for $\gamma \ge 1$. The derivative of $g_6$ is given by $g_6'(\gamma) = k(k-1)\gamma^{k-5} g_7(\gamma)$, where
\begin{align*}
& g_7 (\gamma)\\
= & \left(1296 k^6+3996 k^5+2988 k^4-2100 k^3-3996 k^2-1896 k-288\right) \gamma ^{2 k+4} \\
& + \left(-1296 k^6-5076 k^5-6390 k^4-2892 k^3+78 k^2+384 k+72\right) \gamma ^{2 k+3} \\
& + \left(-24 k^9-344 k^8-1658 k^7-3558 k^6-2674 k^5+3038 k^4+8508 k^3+7848 k^2+3408 k+576\right) \gamma ^{k+4} \\
& + \left(96 k^9+1072 k^8+3864 k^7+6196 k^6+4932 k^5+2272 k^4+636 k^3-396 k^2-528 k-144\right) \gamma ^{k+3} \\
& + \left(-144 k^9-1152 k^8-2588 k^7-1664 k^6+1312 k^5+1608 k^4+92 k^3-280 k^2-64 k\right) \gamma ^{k+2} \\
& + \left(96 k^9+464 k^8+216 k^7-1236 k^6-1068 k^5+360 k^4+348 k^3-20 k^2-24 k\right) \gamma ^{k+1} \\
& + \left(-24 k^9-40 k^8+166 k^7+262 k^6-314 k^5-446 k^4+148 k^3+152 k^2-48 k\right) \gamma ^k \\
& + \left(-2 k^9-33 k^8-235 k^7-955 k^6-2461 k^5-4204 k^4-4790 k^3-3520 k^2-1512 k-288\right) \gamma ^4 \\
& + \left(8 k^9+84 k^8+356 k^7+776 k^6+898 k^5+474 k^4+34 k^3+34 k^2+144 k+72\right) \gamma ^3 \\
& + \left(-12 k^9-54 k^8-32 k^7+172 k^6+284 k^5+130 k^4+128 k^3+184 k^2+64 k\right) \gamma ^2 \\
& +  \left(8 k^9-12 k^8-64 k^7+32 k^6+142 k^5+60 k^4-62 k^3+16 k^2+24 k\right) \gamma \\
& -2 k^9+15 k^8-25 k^7-25 k^6+29 k^5+110 k^4-50 k^3-100 k^2+48 k .
\end{align*}
It can be checked that $g_7 (1) = 0$. Therefore, it suffices to show that $g_7'(\gamma) \ge 0$ for $\gamma \ge 1$. The derivative of $g_7$ is given by $g_7'(\gamma) = 2 (k+2) g_8(\gamma)$, where
\begin{align*}
& g_8 (\gamma ) \\
= & \left(1296 k^6+3996 k^5+2988 k^4-2100 k^3-3996 k^2-1896 k-288\right) \gamma ^{2 k+3} \\
& + \left(-1296 k^6-4428 k^5-5148 k^4-2181 k^3+102 k^2+297 k+54\right) \gamma ^{2 k+2} \\
& + \left(-12 k^9-196 k^8-1125 k^7-2845 k^6-2763 k^5+1697 k^4+6936 k^3+7068 k^2+3264 k+576\right) \gamma ^{k+3} \\
& + \left(48 k^9+584 k^8+2372 k^7+4150 k^6+3460 k^5+1614 k^4+498 k^3-240 k^2-378 k-108\right) \gamma ^{k+2} \\
& + \left(-72 k^9-576 k^8-1294 k^7-832 k^6+656 k^5+804 k^4+46 k^3-140 k^2-32 k\right) \gamma ^{k+1} \\
& + \left(48 k^9+184 k^8-28 k^7-454 k^6-244 k^5+134 k^4+86 k^3-8 k^2-6 k\right) \gamma ^k \\
& + \left(-12 k^9+4 k^8+75 k^7-19 k^6-119 k^5+15 k^4+44 k^3-12 k^2\right) \gamma ^{k-1} \\
& + \left(-4 k^8-58 k^7-354 k^6-1202 k^5-2518 k^4-3372 k^3-2836 k^2-1368 k-288\right) \gamma ^3 \\
& + \left(12 k^8+102 k^7+330 k^6+504 k^5+339 k^4+33 k^3-15 k^2+81 k+54\right) \gamma ^2 \\
& +  \left(-12 k^8-30 k^7+28 k^6+116 k^5+52 k^4+26 k^3+76 k^2+32 k\right) \gamma \\
& + 4 k^8-14 k^7-4 k^6+24 k^5+23 k^4-16 k^3+k^2+6 k .
\end{align*}
It can be checked that $g_8 (1) = 0$. Therefore, it suffices to show that $g_8'(\gamma) \ge 0$ for $\gamma \ge 1$. The derivative of $g_8$ is given by $g_8'(\gamma) = (k+1) g_9(\gamma)$, where
\begin{align*}
& g_9 (\gamma) \\
= & \left(2592 k^6+9288 k^5+8676 k^4-3912 k^3-10380 k^2-5400 k-864\right) \gamma ^{2 k+2} \\
& + \left(-2592 k^6-8856 k^5-10296 k^4-4362 k^3+204 k^2+594 k+108\right) \gamma ^{2 k+1} \\
& + \left(-12 k^9-220 k^8-1493 k^7-4727 k^6-6571 k^5-21 k^4+12048 k^3+15828 k^2+8640 k+1728\right) \gamma ^{k+2} \\
& + \left(48 k^9+632 k^8+2908 k^7+5986 k^6+5774 k^5+2760 k^4+966 k^3-210 k^2-648 k-216\right) \gamma ^{k+1} \\
& + \left(-72 k^9-576 k^8-1294 k^7-832 k^6+656 k^5+804 k^4+46 k^3-140 k^2-32 k\right) \gamma ^k \\
& + \left(48 k^9+136 k^8-164 k^7-290 k^6+46 k^5+88 k^4-2 k^3-6 k^2\right) \gamma ^{k-1} \\
& + \left(-12 k^9+28 k^8+43 k^7-137 k^6+37 k^5+97 k^4-68 k^3+12 k^2\right) \gamma ^{k-2} \\
& + \left(-12 k^7-162 k^6-900 k^5-2706 k^4-4848 k^3-5268 k^2-3240 k-864\right) \gamma ^2 \\
& +  \left(24 k^7+180 k^6+480 k^5+528 k^4+150 k^3-84 k^2+54 k+108\right) \gamma \\
& -12 k^7-18 k^6+46 k^5+70 k^4-18 k^3+44 k^2+32 k .
\end{align*}
It can be checked that $g_9 (1) = 0$. Therefore, it suffices to show that $g_9'(\gamma) \ge 0$ for $\gamma \ge 1$. The derivative of $g_9$ is given by $g_9'(\gamma) = g_{10}(\gamma)$, where
\begin{align*}
& g_{10} (\gamma) \\
=&  \left(5184 k^7+23760 k^6+35928 k^5+9528 k^4-28584 k^3-31560 k^2-12528 k-1728\right) \gamma ^{2 k+1} \\
& + \left(-5184 k^7-20304 k^6-29448 k^5-19020 k^4-3954 k^3+1392 k^2+810 k+108\right) \gamma ^{2 k} \\
& + \big(-12 k^{10}-244 k^9-1933 k^8-7713 k^7-16025 k^6-13163 k^5+12006 k^4+39924 k^3+40296 k^2 \\
& \qquad+19008 k+3456\big) \gamma ^{k+1} \\
& + \left(48 k^{10}+680 k^9+3540 k^8+8894 k^7+11760 k^6+8534 k^5+3726 k^4+756 k^3-858 k^2-864 k-216\right) \gamma ^k \\
& + \left(-72 k^{10}-576 k^9-1294 k^8-832 k^7+656 k^6+804 k^5+46 k^4-140 k^3-32 k^2\right) \gamma ^{k-1} \\
& + \left(48 k^{10}+88 k^9-300 k^8-126 k^7+336 k^6+42 k^5-90 k^4-4 k^3+6 k^2\right) \gamma ^{k-2} \\
& + \left(-12 k^{10}+52 k^9-13 k^8-223 k^7+311 k^6+23 k^5-262 k^4+148 k^3-24 k^2\right) \gamma ^{k-3} \\
& +  \left(-24 k^7-324 k^6-1800 k^5-5412 k^4-9696 k^3-10536 k^2-6480 k-1728\right) \gamma \\
& + 24 k^7+180 k^6+480 k^5+528 k^4+150 k^3-84 k^2+54 k+108 .
\end{align*}
It can be checked that $g_{10} (1) = 350 (k-1) k^2 (k+1) (k+2)^2 > 0$. Therefore, it suffices to show that $g_{10}'(\gamma) \ge 0$ for $\gamma \ge 1$. The derivative of $g_{10}$ is given by $g_{10}'(\gamma) =  g_{11}(\gamma)$, where
\begin{align*}
& g_{11} (\gamma) \\
= & \left(10368 k^8+52704 k^7+95616 k^6+54984 k^5-47640 k^4-91704 k^3-56616 k^2-15984 k-1728\right) \gamma ^{2 k} \\
& + \left(-10368 k^8-40608 k^7-58896 k^6-38040 k^5-7908 k^4+2784 k^3+1620 k^2+216 k\right) \gamma ^{2 k-1} \\
& + \big(-12 k^{11}-256 k^{10}-2177 k^9-9646 k^8-23738 k^7-29188 k^6-1157 k^5 \\
& \qquad +51930 k^4+80220 k^3+59304 k^2+22464 k+3456\big) \gamma ^k \\
& + \big(48 k^{11}+680 k^{10}+3540 k^9+8894 k^8+11760 k^7+8534 k^6+3726 k^5+756 k^4-858 k^3-864 k^2 \\
&\qquad -216 k\big) \gamma ^{k-1} \\
& + \left(-72 k^{11}-504 k^{10}-718 k^9+462 k^8+1488 k^7+148 k^6-758 k^5-186 k^4+108 k^3+32 k^2\right) \gamma ^{k-2} \\
& + \left(48 k^{11}-8 k^{10}-476 k^9+474 k^8+588 k^7-630 k^6-174 k^5+176 k^4+14 k^3-12 k^2\right) \gamma ^{k-3} \\
& + \left(-12 k^{11}+88 k^{10}-169 k^9-184 k^8+980 k^7-910 k^6-331 k^5+934 k^4-468 k^3+72 k^2\right) \gamma ^{k-4} \\
& -24 k^7-324 k^6-1800 k^5-5412 k^4-9696 k^3-10536 k^2-6480 k-1728 .
\end{align*}
It can be checked that $g_{11} (1) = 350 (k-1) k^2 (k+1) (k+2)^2 (9k+5) > 0$. Therefore, it suffices to show that $g_{11}'(\gamma) \ge 0$ for $\gamma \ge 1$. The derivative of $g_{11}$ is given by $g_{11}'(\gamma) = k \gamma^{k-5} g_{12}(\gamma)$, where
\begin{align*}
& g_{12} (\gamma) \\
= & \left(20736 k^8+105408 k^7+191232 k^6+109968 k^5-95280 k^4-183408 k^3-113232 k^2-31968 k-3456\right) \gamma ^{k+4} \\
& + \left(-20736 k^8-70848 k^7-77184 k^6-17184 k^5+22224 k^4+13476 k^3+456 k^2-1188 k-216\right) \gamma ^{k+3} \\
& + \big(-12 k^{11}-256 k^{10}-2177 k^9-9646 k^8-23738 k^7-29188 k^6-1157 k^5+51930 k^4+80220 k^3 \\
& \qquad +59304 k^2+22464 k+3456\big) \gamma ^4 \\
& + \big(48 k^{11}+632 k^{10}+2860 k^9+5354 k^8+2866 k^7-3226 k^6-4808 k^5-2970 k^4-1614 k^3-6 k^2 \\
& \qquad +648 k+216\big) \gamma ^3 \\
& + \left(-72 k^{11}-360 k^{10}+290 k^9+1898 k^8+564 k^7-2828 k^6-1054 k^5+1330 k^4+480 k^3-184 k^2-64 k\right) \gamma ^2 \\
& +  \left(48 k^{11}-152 k^{10}-452 k^9+1902 k^8-834 k^7-2394 k^6+1716 k^5+698 k^4-514 k^3-54 k^2+36 k\right) \gamma \\
& -12 k^{11}+136 k^{10}-521 k^9+492 k^8+1716 k^7-4830 k^6+3309 k^5+2258 k^4-4204 k^3+1944 k^2-288 k .
\end{align*}
It can be checked that $g_{12} (1) = 14 (k-1)k(k+1)(k+2) (  1081 k^3 + 2951 k^2 + 1664^k+ 370 )>0$. Therefore, it suffices to show that $g_{12}'(\gamma) \ge 0$ for $\gamma \ge 1$. The derivative of $g_{12}$ is given by $g_{12}'(\gamma) = 2 (2k+1) g_{13}(\gamma)$, where
\begin{align*}
& g_{13}(\gamma) \\
= & \left(5184 k^8+44496 k^7+130968 k^6+153240 k^5+9528 k^4-145896 k^3-138768 k^2-51840 k-6912\right) \gamma ^{k+3} \\
& + \left(-5184 k^8-30672 k^7-57096 k^6-33636 k^5+9486 k^4+15294 k^3+2574 k^2-1242 k-324\right) \gamma ^{k+2} \\
& + \big(-12 k^{10}-250 k^9-2052 k^8-8620 k^7-19428 k^6-19474 k^5+8580 k^4+47640 k^3+56400 k^2 \\
& \qquad +31104 k+6912\big) \gamma ^3 \\
& + \left(36 k^{10}+456 k^9+1917 k^8+3057 k^7+621 k^6-2730 k^5-2241 k^4-1107 k^3-657 k^2+324 k+324\right) \gamma ^2 \\
& +  \left(-36 k^{10}-162 k^9+226 k^8+836 k^7-136 k^6-1346 k^5+146 k^4+592 k^3-56 k^2-64 k\right) \gamma \\
& + 12 k^{10}-44 k^9-91 k^8+521 k^7-469 k^6-364 k^5+611 k^4-131 k^3-63 k^2+18 k .
\end{align*}
It can be checked that $g_{13} (1) = 14 (k-1)k(k+1)(k+2) ( 687 k^3 + 2516 k^2 + 2490 k + 775 )>0$. Therefore, it suffices to show that $g_{13}'(\gamma) \ge 0$ for $\gamma \ge 1$. The derivative of $g_{13}$ is given by $g_{13}'(\gamma) = 2 g_{14}(\gamma)$, where
\begin{align*}
& g_{14}(\gamma) \\
= & \big(2592 k^9+30024 k^8+132228 k^7+273072 k^6+234624 k^5-58656 k^4-288228 k^3-234072 k^2-81216 k \\
& \qquad -10368\big) \gamma ^{k+2} \\
& + \big(-2592 k^9-20520 k^8-59220 k^7-73914 k^6-28893 k^5+17133 k^4+16581 k^3+1953 k^2-1404 k \\
& \qquad -324\big) \gamma ^{k+1} \\
& + \big(-18 k^{10}-375 k^9-3078 k^8-12930 k^7-29142 k^6-29211 k^5+12870 k^4+71460 k^3+84600 k^2 \\
& \qquad +46656 k+10368\big) \gamma ^2 \\
& +  \left(36 k^{10}+456 k^9+1917 k^8+3057 k^7+621 k^6-2730 k^5-2241 k^4-1107 k^3-657 k^2+324 k+324\right) \gamma \\
& -18 k^{10}-81 k^9+113 k^8+418 k^7-68 k^6-673 k^5+73 k^4+296 k^3-28 k^2-32 k .
\end{align*}
It can be checked that $g_{14} (1) = 7 (k-1)k(k+1)(k+2) ( 1208 k^4 + 6663 k^3 + 12 249 k^2 + 9312 k + 2548 )>0$. Therefore, it suffices to show that $g_{14}'(\gamma) \ge 0$ for $\gamma \ge 1$. The derivative of $g_{14}$ is given by $g_{14}'(\gamma) = 3 (k+1)(k+2)(k+3)(2k+3) g_{15}(\gamma)$, where
\begin{align*}
& g_{15}(\gamma) \\
= & \left(432 k^6+2628 k^5+3696 k^4-412 k^3-3720 k^2-2240 k-384\right) \gamma ^{k+1} \\
& + \left(-432 k^6-612 k^5-60 k^4+221 k^3+66 k^2-17 k-6\right) \gamma ^k \\
& +  \left(-6 k^6-80 k^5-306 k^4-280 k^3+360 k^2+768 k+384\right) \gamma  \\
& + 6 k^6+31 k^5-33 k^4+2 k^3-3 k^2-9 k+6 .
\end{align*}
It can be checked that $g_{15} (1) = 7 (k-1)k(k+1)( 281 k^2 + 471 k + 214) > 0$. Therefore, it suffices to show that $g_{15}'(\gamma) \ge 0$ for $\gamma \ge 1$. The derivative of $g_{15}$ is given by $g_{15}'(\gamma) = g_{16}(\gamma)$, where
\begin{align*}
& g_{16}(\gamma) \\
= & \left(432 k^7+3060 k^6+6324 k^5+3284 k^4-4132 k^3-5960 k^2-2624 k-384\right) \gamma ^k \\
& + \left(-432 k^7-612 k^6-60 k^5+221 k^4+66 k^3-17 k^2-6 k\right) \gamma ^{k-1} \\
& -6 k^6-80 k^5-306 k^4-280 k^3+360 k^2+768 k+384 .
\end{align*}
It can be checked that $g_{16} (1) = (k-1)k ( 2442 k^4 + 8626 k^3 + 11 825 k^2 + 7479 k + 1862 ) > 0$. Therefore, it suffices to show that $g_{16}'(\gamma) \ge 0$ for $\gamma \ge 1$. The derivative of $g_{16}$ is given by $g_{16}'(\gamma) = (k-1)k (3k+1)(3k+2)(4k+3) \gamma^{k-2} g_{17}(\gamma)$, where
\begin{align*}
g_{17}(\gamma) = &\, \left(12 k^3+76 k^2+128 k+64\right) \gamma -12 k^3+4 k^2+3 k-1 \\
\ge& \, 12 k^3+76 k^2+128 k+64 -12 k^3+4 k^2+3 k-1 > 0.
\end{align*}
Deducing backward, we have that $g(\gamma) = (k+1)^3 (\gamma - 1)^2 \gamma^k g_0 (\gamma)\ge 0$ for any $\gamma \ge 1$ and integer $k \ge 5$. This completes the proof of Lemma~\ref{ineq:gg}.
\end{appendix}
\end{document}